\documentclass{amsart}
\usepackage{amssymb,amsmath,latexsym}      
\usepackage{times}

\newcommand{\showhide}[1]{#1} 
\showhide{\usepackage{tikz}}

\numberwithin{equation}{section}

\newtheorem{thm}{Theorem}[section]
\newtheorem{dfn}[thm]{Definition}
\newtheorem{cor}[thm]{Corollary}

\newtheorem{prop}[thm]{Proposition}
\newtheorem{lemma}[thm]{Lemma}
\newtheorem{ids}[thm]{Identities}
\newtheorem{id}[thm]{Identity}
\newtheorem{formula}[thm]{Formula}

\newcommand{\cl}{\mathrm{cl}}

\newcommand{\rk}{\mathrm{rk}}

\title{The $\bar\gamma$-frame for Tutte polynomials of matroids}
\author{Joseph P.S. Kung}

\begin{document}

\begin{abstract}    
Specializing the $\gamma$-basis for the $\binom {n}{r}$-dimensional vector space $\mathcal{G}(n,r)$ spanned by the set of symbols on bit sequences with $r$ $1$'s and $n-r$ $0$'s, we obtain a frame or spanning set for the vector space $\mathcal{T}(n,r)$ spanned by Tutte polynomials of matroids having rank $r$ and size $n$.   Every Tutte polynomial can be expanded as a linear combination with non-negative integer coefficients of elements in this frame.   
We give explicit formulas for the elements in this frame.   These formulas combine to give an expansion of the Tutte polynomial with coefficients obtained by summing numerical invariants over all flats with a given rank and size. 
\end{abstract}

\maketitle


\maketitle


\section{Introduction} \label{Intro} 

Let me begin informally.  I am one of those people who like to stare at tables with numbers, in this particular case, tables of Tutte polynomials.  The polynomials are usually given as a ``tableau'', with the coefficient of $x^iy^j$ as the $ij$th entry.   Here is $T(M(K_7);x,y)$, the Tutte polynomial of the cycle matroid of the complete graph on $7$ vertices:  


{\small
\[
\begin{array}{cccccccccccccccccc} 
&120&490&945&1225&1260& 1120 & 895 & 645 & 420 & 245 &  126& 56 & 21&  6 & 1 
\\
120&644&1225&1330& 1085 & 756& 469 & 245 & 105 & 35&7
\\
274 & 721 & 700 &420 &210 &84 & 21
\\
225 & 280 & 105 &35 
\\
85 & 35 
\\
15
\\1
\end{array} 
\]    
}

The one feature that shouts out is that reading the rows from the right till we reach a corner, we have $126,56,21,6,1$ for the top row, $245,105,35,7$, that is, $7$ times $35,15,5,1$, for the second row, $21,84,210$, that is, $21$ times $10,4,1$, for the third, $105,35$, that is, $35$ times $3,1$, for the fourth, so that these numbers are binomial coefficients multiplied by the number of flats of size $21, 15, 10$, and $6$.  
Everybody has observed this, but there are patterns like this occurring throughout the tableau, only that they are hidden because we only see their linear combinations.  
These patterns suggest that the Tutte polynomial itself is a linear combination (with coefficients having interpretations as numerical invariants summed over flats of a fixed size) of simple subtableaux composed of binomial coefficients.   To explain this is the purpose of the paper.  

There will be easy but very tedious calculations with binomial coefficients.  With old age, I don't remember things well and so these calculations are recorded in great detail.   
The duller parts can always be skipped.   Many properties of the Tutte polynomial do not show themselves fully for small matroids and so the examples in this paper must necessarily be large; in fact, it could be argued that they are not large enough.  
Nevertheless, the final formula is quite simple, and it can be found in Section \ref{FlatTutte}.   
\vskip 0.1in 
{\center *** \endcenter}
\vskip 0.1in
The objective of this paper is to describe an expansion of the Tutte polynomial of a matroid as a linear combination {\it with non-negative integer coefficients} of a set of polynomials collectively called the $\bar\gamma$-frame.  Except when the rank is one or zero, the elements of the $\bar\gamma$-frame are not Tutte polynomials of matroids, but behave formally like a rational multiple 
of the Tutte polynomial of a perfect matroid design.  While they are not exactly simple to work with, they have enough structure to yield workable formulas.  
The $\bar\gamma$-expansion can be modified (by a change in the order of summation) to an expansion of the Tutte polynomials with coefficients which are sums of M\"obius invariants over all flats of a given rank and size.  

We begin with a brief description in Section \ref{Ginvariant} of the $\mathcal{G}$-invariant and describe how it specializes to the Tutte polynomial.  The $\mathcal{G}$-invariant is defined over bit sequences which are naturally ordered as a sublattice of the Young lattice.  We describe this partial order in Section \ref{Younglattice}.  In Section \ref{FB}, we describe the $\gamma$-basis for the vector space $\mathcal{G}(n,r)$ spanned by symbols of bit sequences with $r$ $1$'s and $n-r$ $0$'s in detail and apply the Tutte-polynomial specialization to the basis elements to obtain the $\bar\gamma$-frame.  In Section \ref{slicenorm}  and \ref{formula}, we give an explicit formula for the $\bar\gamma$-frame elements.  
Using this formula, we give a formula for the Tutte polynomial based on numbers derived from summing M\"obius-function invariants over flats of a given size and rank in Section \ref{FlatTutte}. 

\section{Specializing the $\mathcal{G}$-invariant to the Tutte polynomial} \label{Ginvariant}

This paper is based on ideas from the theory of valuative invariants on matroid base polytopes (see \cite{FalkKung} for a short informal introduction).   We will use methods from 
\cite{Catdata, syzygy} in some proofs (especially in Section \ref{Meowbius}), but the earlier sections should be accessible.  We begin with a brief summary of the basic theory.  
 
The $\mathcal{G}$-invariant was introduced by Derksen \cite{Derksen} in 2009.     
Let $M$ be an $(n,r)$-matroid, that is, a rank-$r$ matroid 
on the set $\{1,2,\ldots, n\}$ with rank function $\rk$ and closure $\cl.$  For a permutation $\pi$ on $\{1,2, \ldots, n\},$  the {\em rank sequence} $\underline{r}(\pi)$ of $\pi$ is the sequence   $ r_1 r_2 \ldots r_n$ 
defined by $r_1 = \rk(\{\pi(1)\})$ and for $j \geq 2,$
\[
r_j = \rk(\{\pi(1), \pi(2),\ldots,\pi(j)\}) - \rk(\{\pi(1), \pi(2),\ldots,\pi(j-1)\}). 
\]
For a matroid, $r_j = 0$ or $1,$ and there are exactly $r$ $1$'s.  

A {\sl bit sequence} is a sequence of zeros and ones, and an \emph{$(n,r)$-sequence} is a bit sequence of length $n$ with (exactly) $r$ $1$'s.  
Using the notation where $1^a$ stands for a sequence of $a$ (consecutive) $1$'s and $0^b$ a sequence of  $b$ $0$'s,     
we will often display a bit sequence in the following way,
\[
0^{a_0} 1_1 0^{a_1-1} 1_2 0^{a_2-1} \ldots  1_i0^{a_i-1} \ldots 1_{r-1}0^{a_{r-1}-1} 1_r 0^{a_r-1},
\]     
with a subscript $i$ on the $i$th $1$-bit.  We use this notation so that we can easily refer to the $i$th $1$-bit;  $1_i$ is always $1$.   

An \emph{$(n,r)$-composition} $\underline{a}$ is a sequence $a_0, a_1, a_2, \ldots , a_r$  of integers, with $a_0 \geq 0$,  $a_i > 0, 1 \leq i \leq r$, and 
\[
a_0 + a_1 + a_2 + \cdots  + a_r = n.  
\]
Bit sequences and compositions carry the same information under the bijection   
\[
0^{a_0} 1_1 0^{a_1 -1}1_2 0^{a_2 -1} \ldots 1_r 0^{a_r -1} \,\,\longleftrightarrow \,\,  a_0,a_1,a_2, \ldots,a_r. 
\] 
There is a third way to show the information in a bit sequence.  The \emph{partial sums} $\xi_i$ are defined by 
\[
 \xi_{i} =a_0+a_1+a_2+\ldots+a_i.
\]
\noindent   
We shall freely switch between these three ways of displaying the same information without comment.  

Let $[\underline{r}]$ be a variable or formal symbol, one for each $(n,r)$-sequence $\underline{r},$ and $\mathcal{G}(n,r)$ be the vector space of dimension $\binom {n}{r}$ consisting of all formal linear combinations  
of \emph{symbols} $[\underline{r}]$ with coefficients over a field $\mathbb{K}$ of characteristic zero.     By construction, the symbols form a basis, called the \emph{symbol basis}. 
The {\em $\mathcal{G}$-invariant} and its coefficients $g_{\underline{r}}$ are defined by 
\[
\mathcal{G}(M) = \sum_{\pi}  [\underline{r}(\pi)] = \sum_{\underline{r}} g_{\underline{r}}(M) [\underline{r}],
\]
where the first sum ranges over all $n!$ permutations of $\{1,2,\ldots,n\}.$ 
A {\em specialization} of the $\mathcal{G}$-invariant taking values in a $\mathbb{K}$-vector space $V$ is a linear map from $\mathcal{G}(n,r)$ to $V$.  Given a basis of $\mathcal{G}(n,r)$, we can specify a specialization by assigning a vector in $V$ to each basis element and extending by linearity.  In particular, a specialization is specified if we assign a value in $V$ to each symbol $\underline{r}$.
There are more general notions of specializations; for example, one could define $\mathcal{G}(n,r)$ as a free $\mathbb{Z}$-module and specialize through a $\mathbb{Z}$-linear map to an abelian group $\mathbb{A}$.   

It follows from a theorem of Derksen and Fink \cite{DerksenFink} that the Tutte polynomial is a specialization of the $\mathcal{G}$-invariant.  
To clarify notation, recall that if $M$ is a rank-$r$ matroid on the set $S$, the  \emph{Tutte polynomial} $T(M)$ and its \emph{coefficients} $t_{ij}(M)$ are defined by   
\begin{eqnarray*}
T(M) = T(M;x,y) &=& \sum_{A \subseteq S}  (x-1)^{r - \rk (A)} (y-1)^{|A| - \rk(A)}
\\
&=& \sum_{i,j} t_{ij}(M)  x^i y^j.
\end{eqnarray*} 
Surveys of the Tutte polynomial can be found in  \cite{BryOx, Handbook} and we shall assume some familiarity with the area.  As a starting point, we note the easy fact that 
\begin{align*} 
&T(U_{r,n}) = (x-1)^r + \binom {n}{1}(x-1)^{r-1} + \cdots + \binom {n}{r-1}(x-1)  + \binom {n}{r}
\\
&\qquad\qquad\qquad + \binom {n}{r+1} (y-1) + \cdots + \binom {n}{n-1}(y-1)^{n-r-1}+(y-1)^{n-r}, 
\end{align*} 
where $U_{r,n}$ is a uniform matroid.  This explicit formula was in fact our motivation and it will play a surprisingly fundamental role in the theory.  

We denote by $\mathcal{T}(n,r)$ the subspace in the algebra $\mathbb{K}[x,y]$ of polynomials in the variables $x$ and $y$ with coefficients in the field $\mathbb{K}$ spanned by the Tutte polynomials of $(n,r)$-matroids.    
The specialization sending the $\mathcal{G}$-invariant to the Tutte polynomial is given explicitly in the following lemma (Derksen \cite{Derksen}).  

\begin{lemma} \label{specialization}    The linear transformation $\mathsf{Sp}:\mathcal{G}(n,r) \to \mathcal{T}(n,r)$ given by linearly extending the assignment 
\[
[r_1r_2 \ldots r_n] \mapsto   \frac {1}{n!}\sum_{m=0}^n  
\binom {n}{m} (x-1)^{r - \mathrm{wt}(r_1r_2 \ldots r_m)}(y-1)^{m-\mathrm{wt}(r_1r_2 \ldots r_m)} ,
\] 
where the Hamming weight $\mathrm{wt}(r_1r_2 \ldots r_m)$ is the number of $1$'s in the initial segment $r_1r_2 \ldots r_m,$ sends the $\mathcal{G}$-invariant of a matroid to its Tutte polynomial.   
\end{lemma}

We will use the following notation for \emph{specialized symbols}:  
\[
[\underline{b};x,y] = \mathsf{Sp}[\underline{b}].   
\]
A way to do the specialization $[\underline{b}] \to 
[\underline{b};x,y]$ is to write down the row of binomial coefficients $\binom {n}{r}$, multiply the extreme left binomial coefficient $\binom {n}{0}$ by $(x-1)^r(y-1)^0,$ go one position to the right, and repeat, decreasing 
the exponent on $(x-1)$ by $1$ if we go pass a $1$, and increasing the exponent of $(y-1)$ by $1$ if we go pass a $0$, adding up all the terms, and finally, putting in the factor of $1/n!$.  
For example, 
\begin{align*}
[11001;x,y] &= \tfrac {1}{120}\bigl((x-1)^3 + 5(x-1)^2 
\\
& \qquad + 10(x-1) + 10(x-1)(y-1) + 5(x-1)(y-1) + (y-1)^2\bigr),
\end{align*} 
\begin{align*}
&[11010;x,y] 
\\
& = \tfrac {1}{120}\bigl((x-1)^3 + 5(x-1)^2 + 10(x-1) + 10(x-1)(y-1) + 5(y-1) + (y-1)^2\bigr),
\end{align*} 
and
\[
[11100;x,y]
 = \tfrac {1}{120}\bigl((x-1)^3 + 5(x-1)^2 + 10(x-1) + 10 + 5(y-1) + (y-1)^2\bigr).
\] 
The third specialization is an instance of the fact that 
\[
[1^r0^{n-r};x,y] = \tfrac {1}{n!} T(U_{r,n}), 
\eqno(2.1) \]
where $U_{r,n}$ is a uniform matroid.  
Taking the difference of the second and third specializations, 
we have 
\[
[11010;x,y] = [11100;x,y] + \frac {xy -x-y}{3! \, 2!}.
\]
This is an example of the following lemma, which is the first step towards the syzygy theorem \cite{syzygy} for the specialization $\mathsf{Sp}$. 

\begin{lemma}\label{syzygy}  
Let 
\[
\underline{b} = b_1b_2 \ldots b_{m-1}01_{k+1}b_{m+2} \ldots b_n, \,\, \,\, \underline{b}^\prime  = b_1b_2 \ldots b_{m-1}1_{k+1} 0b_{m+2} \ldots b_n, 
\]
so that $\underline{b}$ and $\underline{b}^\prime$ are $(n,r)$-sequences differing only in positions $m$ and $m+1$.  
Then 
\[
[\underline{b};x,y] =  [\underline{b}^\prime;x,y ] +
\frac { (xy - x -y)(x-1)^{r-k-1}  (y-1)^{m-k-1}}{m! (n-m)!}.
\]
\end{lemma}
\noindent 
The difference $[\underline{b};x,y] - [\underline{b}^\prime;x,y ]$, that is, the polynomial 
\[
\frac {1} {m! (n-m)!}(xy - x -y)(x-1)^{r-k-1}  (y-1)^{m-k-1}
\] 
is called the \emph{syzygy term} for the \emph{move  of $1_{k+1}$ from position $m+1$ to $m$.}
It depends only on $k$ and $m$.  
We can iterate Lemma \ref{syzygy}, moving every $1$-bit as far as possible to the left.  

\begin{cor} \label{syzygyiterated} A specialized symbol 
can be written as a linear combination of $T(U_{r,n})$ and syzygy terms.  In the case when $a_0=0$, $[1_10^{a_1-1}1_20^{a_2-1} \ldots 1_r0^{a_r-1}]$ equals
\[
\frac {1}{n!}\left(T(U_{r,n}) + (xy-x-y)\sum_{k=1}^{r-1}  (x-1)^{r-k-1} \sum_{\alpha=0}^{\xi_{k}-k-1} \binom {n}{\xi_k -\alpha} (y-1)^{\xi_{k} -k-1-\alpha}\right). 
\]
\end{cor}
\noindent 
We will not need the exact formula, but there is no harm in stating it. 

We next introduce the polynomials 
$\tau(d,\alpha;y)$.   These polynomials play an essential role in this paper.  

\begin{dfn} \label{taupolynomials}  
For $d \leq 0,$ let $\tau(d,\alpha;y) = 0$, and for $d \geq 1$ and $\alpha \geq 0$,  let
\begin{align*} 
\tau (d,\alpha; y) &= \sum_{i=0}^{\alpha}   \binom {\alpha + d - 1 - i}{\alpha-i} y^{i} 
\\
&=  \binom {\alpha+d-1}{\alpha} + \binom {\alpha+d-2}{\alpha-1} y  + \cdots +  \binom {d+1}{2}y^{\alpha-2} + \binom {d}{1} y^{\alpha-1}  + \binom {d-1}{0}y^{\alpha}.
\end{align*}

\end{dfn} 
For example, 
\[
\tau (5,4;y) =    \tbinom {8}{4} + \tbinom {7}{3} y +  \tbinom {6}{2} y^2 + \tbinom {5}{1} y^3 + y^4 = 70+ 35y +  15y^2 + 5y^3 + y^4.   
\]
These polynomials can also be defined by $\tau(d,\alpha;y)=0$ when $\alpha \leq 0$, $\tau (0,\alpha,0; y) = 1$, and the recursion, 
\[
\tau (d+1,\alpha; y) = \binom {d+ \alpha}{\alpha+1} + y \tau (d, \alpha; y).  
\]
From the recursion (or other arguments), we have the following identity.  It shows that $\tau(d, \alpha;y)$ are truncated (or manx-tailed) binomial series. 

\begin{id} \label{tauinadifferentform} 
\[
\tau (d,\alpha;y) = (y-1)^\alpha + \binom {\alpha+ d}{1} (y-1)^{\alpha-1} + \cdots +\binom {\alpha+d}{\alpha-1}(y-1) + \binom {\alpha + d}{\alpha}.  
\]
\end{id} 
For example, 
\begin{align*}
\tau (5,4;y)  &=    (y-1)^4 + \tbinom {9}{1}(y-1)^3 + \tbinom {9}{2} (y-1)^2 +  \tbinom {9}{3} (y-1) + \tbinom {9}{4}  
\\
&= (y-1)^4 + 9(y-1)^3 + 36 (y-1)^2 +  84 (y-1) + 126.  
\end{align*} 
Identity \ref{tauinadifferentform} yields the following formula for Tutte polynomials of uniform matroids.  

\begin{lemma} \label{uniform0}
$ T(U_{r,n}) 
=  \tau(r,n-r;x) + \tau(n-r,r;y) .$ 
\end{lemma} 

\noindent  In particular,  $\tau(d,\alpha;y)$ is the Tutte polynomial evaluation $T(U_{d, \alpha+d};0,y)$ and 
$d$ can be thought of as a rank and $\alpha$ as a nullity (or the other way around).  

We next present some elementary binomial-coefficient identities which we shall use repeatedly.  

\begin{id} \label{binomialID0}
\begin{align*}
 \sum_{i=0}^j  \binom {a + i - 1}{i}= \binom {a-1}{0} +\binom {a}{1} +\binom {a+1}{2} + \cdots + \binom {a+j-1}{j} = \binom {a+j}{j} 
\end{align*}
\end{id} 

\begin{id} \label{binomialID}
\begin{align*}
 \sum_{i=0}^j (-1)^i \binom {n}{i} = \binom {n}{0} -\binom {n}{1} +\binom {n}{2} + \cdots + (-1)^{j} \binom {n}{j} = (-1)^{j} \binom {n-1}{j} 
\end{align*} 
\end{id} 

\begin{ids} \label{binomialID1} 
\begin{align*} 
& \binom {a + j - 1}{j}  \binom {A + a + j - 1}{a+j-1} =  \binom {A + a  - 1}{a-1} \binom {A +j - 1}{j} \, ,
\\
&
\sum_{j=0}^\beta  \binom {a+j-1}{j}\binom {A + a +j-1}{a+j-1} = \binom {A + a-1}{a-1} \binom {A + a + \beta}{\beta}. 
\end{align*} 
\end{ids}

\begin{proof}  To prove the first identity, note that both sides equal
\[
\frac  {1}{j!(a-1)!} 
\cdot  \frac  {(A + a  +j - 1)!}{A!}.
\]
Using the first identity, we have 
\begin{align*}
\sum_{{j}=0}^{\beta} \binom {a +j - 1}{j}  \binom {A + a + j - 1}{a +j-1}
& =  \binom {A + a - 1}{a-1}  \sum_{j=0}^{\beta} \binom {A + a +j - 1}{j} 
\\
& = \binom {A + a - 1}{a-1}  \binom {A + a + \beta}{\beta} \, .
\end{align*}
In the last step, we use Identity \ref{binomialID0}. 
\end{proof}

We shall often (but not always) use the notation 
\[
\underline{a}! = a_0! a_1! a_2! \cdots a_r!  \,,
\] 
\[
a_{i:j} = \left\{ \begin{array}{cc} a_i+a_{i+1} + \cdots + a_{j}  \,  & \mathrm{if}\,\, i \leq j 
\\
a_j+a_{j-1} + \cdots + a_{i}   \,  & \mathrm{if}\,\, j \leq i \, .
\end{array}\right. 
\]
It follows from the definition that $a_{i:j} = a_{j:i}$, $a_{i:i} = a_i$,  $a_{1:k} = \xi_k$, $a_{r:k+1} = n-\xi_k$, and if $i \leq j < k$ or $i \geq j > k$, $a_{i:j}+a_{j+1:k} = a_{i:k}$.

\section{Bit sequences and shift vectors} \label{Younglattice} 

Let $\mathcal{S}(n,r)$ be the set of $(n,r)$-sequences.  We define the (partial) order $\trianglerighteq$ in the following way. If $\underline{r}$ and $\underline{t}$ are two $(n,r)$-sequences, then 
$\underline{t} \trianglerighteq \underline{r}$ if for every index $j,\, 1 \leq j \leq n,$ 
\[
t_1 + t_2 + \cdots + t_j   \geq   r_1 + r_2 + \cdots + r_j,
\] 
in other words, reading from the left, there are always at least as many $1$'s in $\underline{t}$ as there are in $\underline{r}.$    
This order has maximum $1^r0^{n-r}$ and minimum $0^{n-r}1^r$.     

The partial order $(\mathcal{S}(n,r),\trianglerighteq)$ is a sublattice of Young's (partition) lattice (see, for example, \cite[p.~288]{EC2}).      We shall use the following description of 
$\trianglerighteq$.

\begin{lemma} \label{stayright} 
If $\underline{r}$ and $\underline{t}$ are $(n,r)$-sequences, then 
$\underline{t} \trianglerighteq \underline{r}$ if and only if 
the $i$th  $1$-bit in $\underline{t}$ is at the same position or to the left of the $i$th $1$-bit  in $\underline{r}$ but to the right of the $(i-1)$th $1$-bit in $\underline{r}$.  
\end{lemma}

Elements in the $\gamma$-basis are sums over principal filters in 
$(\mathcal{S}(n,r);\trianglerighteq)$.  
Let 
$\underline{a}$ be a bit sequence.  The (principal) filter $[\underline{a})$ is the set of bit sequences $\underline{b}$ such that $\underline{b} \trianglerighteq \underline{a}$.  Every bit 
 sequence $\underline{b}$  in $[\underline{a})$ is obtained by shifting bit $1_i$ to the left by $s_i(\underline{b},\underline{a})$ positions, where
 \[
 s_i(\underline{b},\underline{a}) = a_{0:i} - b_{i:0}.  
 \] 
 Hence, a bit sequence $\underline{b}$ in $[\underline{a})$ is determined by its \emph{shift vector}, defined to be the $r$-dimensional vector 
 \[
 (s_1(\underline{b},\underline{a}) ,s_2(\underline{b},\underline{a}),\ldots,s_r(\underline{b},\underline{a})).
 \]
For example, in the filter $[001_1 00001_2 00001_3 000)$,  the sequence $1_100 0001_20 01_3000000$ has shift vector $(2,1,3)$.   
Reminding ourselves of the restriction in Lemma \ref{stayright} that one cannot shift a $1$-bit pass another,  there is a bijection between bit sequences in $[\underline{a})$ and shift vectors $(s_i(\underline{b},\underline{a}))$.  Thus, we may identify a bit sequence $\underline{b}$ in 
$[\underline{a})$ with its shift vector in $(s_i(\underline{b}, \underline{a}))$.    
Note that shifts are not moves:  in fact, moving $1_{k+1}$ in the bit sequence $\underline{b}$ changes its shift vector, replacing the coordinate $s_{k+1}$ by $s_{k+1}+1$.  

If $\underline{b}$ and $\underline{c}$ are in $[\underline{a})$, then $\underline{b} \trianglerighteq \underline{c}$ if and only if for every index $i$, $s_i(\underline{b},\underline{a}) \geq s_i(\underline{c},\underline{a})$.   
In this way, the filter $[\underline{a})$ is embedded as a subset of 
the \emph{box} $\mathsf{Box}(\underline{a})$, defined by 
\[
\mathsf{Box}(a_0,a_1,\ldots,a_r)  = \prod_{i=0}^{r-1}   \{0,1,\ldots,a_{0:i}-i\}.
\] 
under the  product partial order 
$(s_1,s_2,\ldots,s_r) \geq (t_1,t_2,\ldots,t_r)$ if $s_i \geq t_i$ for all $i, \, 1 \leq i \leq r$. 
We define the \emph{thickness} of the composition $\underline{a}$  to be $r - \iota$, where 
$\iota = 0$ if $a_0 > 0$ and when $a_0 = 0$, $\iota$ is the minimum index such that $a_\iota \geq 2$.    Intuitively, the number $r - \iota+1$ is the dimension of 
$\mathsf{Box}(\underline{a})$ embedded as a subset of $\{0,1,2,\ldots,\}^r$.   

Slices are subsets of filters defined by imposing conditions on coordinates of shift vectors.  
The \emph{basic slices} are those obtained by fixing one coordinate $s_i$, that is, the subsets $[\underline{a}; s_i = \alpha)$ 
consisting of the bit sequences $\underline{b}$ in $[\underline{a})$ 
such that 
$s_i (\underline{b},\underline{a})$ equals a fixed non-negative integer $\alpha$.  We can then take set operations on basic slices to obtain \emph{slices.}  For example, 
$[\underline{a};s_i \leq \alpha)$ is the union of the basic slices $[\underline{a}; s_i = \beta)$, $0 \leq \beta \leq \alpha$, and
$[\underline{a};s_i = \alpha, s_j = \beta)$ is the intersection of the basic slices $[\underline{a};s_i = \alpha)$ and $[\underline{a}; s_j = \beta)$.

\section{Frames and bases}  \label{FB}

The $\gamma$-basis was introduced in \cite{Catdata}.  Let $M$ be an $(n,r)$-matroid and let $\pi$ be a permutation on $\{1,2,\ldots,n\}$.  
In the process of determining a rank sequence $\underline{r}(\pi)$,  $\pi$ also determines  a non-decreasing chain of flats by $X_0 = \cl(\emptyset)$, and for $1 \leq i \leq n,$ 
\[
X_i = \cl(\{ \pi(1), \pi(2), \ldots, \pi(i)\}).
\]
The flat stays the same if the $i$-bit $r_i$ is $0$ and goes up to a flat of one higher rank if $r_i$ equals $1$.  Removing duplicates, one obtains a flag or maximal strict chain of flats called the \emph{flag of the permutation} $\pi$.   Each flag $(X_i)$ determines a composition $a_0 = |\cl(\emptyset)|$ and for $1 \leq i \leq r,$ $a_i = |X_i| - |X_{i-1}|$.
If we fix a flag $(X_i)$ and sum over the rank sequences of all the permutations with that flag, then the sum depends only on the composition $\underline{a}$ of the flag and is independent of the underlying matroid structure.  
We define $\gamma(\underline{a})$ to be that sum, that is, we choose a flag $(X_i)$ with composition $\underline{a}$ in any $(n,r)$-matroid with such a flag,  and set  
\[
\gamma(a_0,a_1,a_2, \ldots, a_r) = \sum_{\pi}       [\underline{r}(\pi)],  
\]   
the sum ranging over all permutations $\pi$ with flag $(X_i)$.   Note that given an $(n,r)$-composition $a_0,a_1,a_2, \ldots,a_r,$ the direct sum 
\[
U_{0,a_0} \oplus U_{1,a_1} \oplus U_{1,a_2} \oplus \cdots \oplus U_{1,a_r} 
\]
of $a_0$ loops and rank-$1$ matroids with $a_i$ parallel elements, contains a flag with composition $a_0,a_1,\ldots,a_r$, so that there is always a matroid in which this can be done.   

The proof that $\gamma(\underline{a})$ as defined is independent of the matroid structure is by showing an explicit formula for it.  This formula, given in \cite{Catdata}, is  
\begin{align*}
&\gamma(a_0,a_1,a_2,\ldots,a_r) = 
\\
&\sum_{b_0,b_1,\ldots,b_r} (a_0)_{b_0} \left(\prod_{j=1}^r  a_j\left( a_j - 1 + \left(\sum_{i=0}^{j-1} a_i - b_i\right)\right)_{b_j-1}   \right) [0^{b_0}10^{b_1-1} \ldots 10^{b_{r}-1}],
\end{align*}
where $(t)_k$ is the falling factorial $t(t-1)\cdots(t-k+1)$ and the sum ranges over compositions $b_0,b_1,\ldots,b_r$ such that 
$0^{b_0}10^{b_1-1} \ldots 10^{b_r-1} \trianglerighteq 0^{a_0}10^{a_1-1} \ldots 10^{a_r-1}$.     
A little more work shows that as $\underline{a}$ ranges over all $(n,r)$-composition, $\gamma(\underline{a})$ form a basis (called the \emph{$\gamma$-basis}) for $\mathcal{G}(n,r)$ and we have the following theorem, giving a combinatorial interpretation for the coefficients when the $\mathcal{G}$-invariant is expanded in terms of the $\gamma$-basis elements.  

\begin{thm} \label{Cat}
Let $M$ be an $(n,r)$-matroid.  Then 
\[
\mathcal{G}(M) = \sum_{\underline{a}}  \nu(M;\underline{a})  \gamma(\underline{a}),
\]
where $\nu(M;a_0,a_1,a_2,\ldots,a_r)$ is the number of flags $X_0, X_1, X_2, \ldots,X_r$ such that $X_{i-1}  \subset X_i$,  $\rk(X_i) = i$, $a_0 = |X_0|$, and $a_i = |X_i| - |X_{i-1}|$ for $1 \leq i \leq r$.
\end{thm}

The horrible-looking sum for $\gamma(\underline{a})$ arises naturally from the proof in \cite{Catdata} of Theorem \ref{Cat}. We rewrite the sum using shift vectors
\footnote{The rewriting is straight-forward; begin by observing that $a_{0:i} - b_{0:i} = s_i(\underline{b},\underline{a})$. } 
 and for the purposes of this paper, take it as the definition of $\gamma(\underline{a})$.  

\begin{dfn} \label{defgamma}   Let $\underline{a} = 0^{a_0}10^{a_1-1} \ldots 10^{a_r-1}$.  Then 
\[
\gamma(a_0,a_1,a_2, \ldots,a_r)
= 
\underline{a}! \, \sum_{\underline{b} \in [\underline{a})}  \mathsf{c}_{\underline{a}}  (s_1(\underline{b},\underline{a}), s_2(\underline{b},\underline{a}), s_{3} (\underline{b},\underline{a}), \ldots, s_{r} (\underline{b},\underline{a})) [\underline{b}],
\]
where the \emph{coefficient function} $\mathsf{c}_{\underline{a}} $ is defined by 
\[
\mathsf{c}_{\underline{a}} (s_1,s_2, \ldots,s_r) = \binom {a_1 +s_1}{s_1} \prod_{i=2}^{r}  \binom {a_i+s_{i}-1 }{s_{i}}  
\]
and
$(s_1(\underline{b},\underline{a}),s_2(\underline{b},\underline{a}),\ldots,s_r(\underline{b},\underline{a}))$ is the shift vector of $\underline{b}$ in the 
filter $[\underline{a}).$
\end{dfn} 

The sum in Definition \ref{defgamma} can be visualized using the Hasse diagram of the filter $[\underline{a})$.   
Some examples might be useful.  Consider first $\gamma(0,1,1,3,5)$.  The composition $0,1,1,3,5$ corresponds to the bit sequence $1110010000$. 
It has thickness $1$ and the filter $[1110010000)$ is a total order with $3$ elements.  We have 
\begin{align*}
\gamma(0,1,1,3,5) &= 0!1!1!3!5! \bigl(\tbinom {6}{2}[1111000000] + \tbinom {5}{1}[1110100000]  + \tbinom {4}{0}[1110010000] \bigr)
\\
&=  720 \bigl(15 [11110^6] + 5[111010^5]  + [1110010^4]\bigr).  
\end{align*}
Next consider $\gamma(0,1,2,3,4)$.  Its composition has thickness $2$ and the filter $[11010010^3)$ is a $7$-element partial order whose Hasse diagram can be discerned in 
the way we display $\gamma(0,1,2,3,4)$: 
\begin{align*}
& \!\begin{array}{cccc}
0!1!2!3!4! \big( \binom{6}{3}  \binom{3}{1}[1^40^6] \! \!\!\!\!&+   \binom{5}{2}  \binom{3}{1}[111010^5]  \!\!\!\! \!\! &+  \binom{4}{1}  \binom{3}{1}[1110010^4] \!\!\!\!\!&+  \binom{3}{0}  \binom{3}{1}[11100010^3] 
\\
\,
\\
 \,\,                     &+    \binom{5}{2}  \binom{2}{0}[110110^5]   \!\!\!\!\!\!&+  \binom{4}{1}  \binom{2}{0}[1101010^4]    \!\!\!\!\!&+  \binom{3}{0}  \binom{2}{0} [11010010^3] \big).
\end{array}
\end{align*}

The last example is $\gamma(0,1,2,3,4,5)$.  The composition has thickness $3$.     
A $3$-dimensional Hasse diagram is tricky to draw but we can visualize $\gamma(0,1,2,3,4)$ as the sum over the union of two basic slices, one at $s_3 = 0$ and the other 
at $s_3 = 1$.   This is shown in Figures \ref{fig:backslice} and \ref{fig:frontslice}.  Four other slices are shown in Figures \ref{fig:bottom} and \ref{fig:rightslices}.   
We use the convention where going up the partial order 
$\trianglerighteq$ is going, more or less, northwest, so that the minimum is at the lower right-hand corner.  
The vertices in the Hasse diagrams are labelled with the term (the symbol with its coefficient) it contributes to the $\gamma$-basis element.   Many labels are omitted for clarity.  

\showhide{
\begin{figure}
  \centering
  \begin{tikzpicture}[scale=.9]  
%
\node[inner sep = 0.3mm] (b11) at (6.5,-1.05) {\small $\diamond$};  \node[inner sep = 0.3mm] (b11c) at (5.3,-1.05) {\small $60[11110^610^4]$};
 \node[inner sep = 0.3mm] (b21) at (7.5,-1.05) {\small $\diamond$};
\node[inner sep = 0.3mm] (b31) at (8.5,-1.05) {\small $\diamond$};
\node[inner sep = 0.3mm] (b41) at (9.5,-1.05) {\small $\diamond$};  \node[inner sep = 0.3mm] (b41c) at (10.7,-1.05) {\small $\,\,\,3 [1110^310^310^4]$};
\foreach \from/\to in {b11/b21,b21/b31,b21/b31,b31/b41} \draw(\from)--(\to);
\node[inner sep = 0.3mm] (b12) at (6.5,-0.05) {\small $\diamond$}; \node[inner sep = 0.3mm] (b11c) at (5.2,-.05) {\small $300 [11110^510^5]$};
 \node[inner sep = 0.3mm] (b22) at (7.5,-0.05) {\small $\diamond$};
\node[inner sep = 0.3mm] (b32) at (8.5,-0.05) {\small $\diamond$};
\node[inner sep = 0.3mm] (b42) at (9.5,-0.05) {\small $\diamond$};  \node[inner sep = 0.3mm] (b11c) at (10.8,-.05) {\small $\,\,15[1110^310^210^5]$};

\foreach \from/\to in {b12/b22,b22/b32,b22/b32,b32/b42} \draw(\from)--(\to);
\node[inner sep = 0.3mm] (b13) at (6.5, .95) {\small $\diamond$}; \node[inner sep = 0.3mm] (b11c) at (5.2,.95) {\small $900 [11110^410^6]$};

 \node[inner sep = 0.3mm] (b23) at (7.5, .95) {\small $\diamond$};
\node[inner sep = 0.3mm] (b33) at (8.5,  .95) {\small $\diamond$};
\node[inner sep = 0.3mm] (b43) at (9.5, .95) {\small $\diamond$};  \node[inner sep = 0.3mm] (b11c) at (10.8,.95) {\small $\,45[1110^31010^6]$};

\foreach \from/\to in {b13/b23,b23/b33,b23/b33,b33/b43} \draw(\from)--(\to);
\node[inner sep = 0.3mm] (b14) at (6.5, 1.95) {\small $\diamond$}; \node[inner sep = 0.3mm] (b11c) at (5.1,1.95) {\small $2100 [11110^310^7]$};
 \node[inner sep = 0.3mm] (b24) at (7.5, 1.95) {\small $\diamond$};
\node[inner sep = 0.3mm] (b34) at (8.5,  1.95) {\small $\diamond$};
\node[inner sep = 0.3mm] (b44) at (9.5, 1.95) {\small $\diamond$};  \node[inner sep = 0.3mm] (b11c) at (10.7,1.95) {\small $70[1110^310^7]$};

\foreach \from/\to in {b14/b24,b24/b34,b24/b34,b34/b44} \draw(\from)--(\to);
\node[inner sep = 0.3mm] (b15) at (6.5,  2.95) {\small $\diamond$}; \node[inner sep = 0.3mm] (b11c) at (5.1,2.95) {\small $4200 [11110^210^8]$};

 \node[inner sep = 0.3mm] (b25) at (7.5,  2.95) {\small $\diamond$};
\node[inner sep = 0.3mm] (b35) at (8.5,    2.95) {\small $\diamond$};\node[inner sep = 0.3mm] (b32c) at (9.9,2.95) {\small $\,840 [1110^2110^8] \,$};
\foreach \from/\to in {b15/b25,b25/b35,b25/b35} \draw(\from)--(\to);
\node[inner sep = 0.3mm] (b16) at (6.5,  3.95) {\small $\diamond$}; \node[inner sep = 0.3mm] (b11c) at (5.2,3.95) {\small $7560 [1111010^9]$};
 \node[inner sep = 0.3mm] (b26) at (7.5,  3.95) {\small $\diamond$};
\node[inner sep = 0.3mm] (b36) at (7.5,  3.95) {\small $\diamond$};
\foreach \from/\to in {b16/b26,b26/b36,b26/b36} \draw(\from)--(\to);
\node[inner sep = 0.3mm] (b17) at (6.5,  3.95) {\small $\diamond$};
 \node[inner sep = 0.3mm] (b27) at (7.5,  3.95) {\small $\diamond$};  \node[inner sep = 0.3mm] (b27c) at (8.9,4.05) {\small $\,3760[11101110^9]$};

\foreach \from/\to in {b17/b27} \draw(\from)--(\to);
\node[inner sep = 0.3mm] (b17) at (6.5, 4.95) {\small $\diamond$}; \node[inner sep = 0.3mm] (b17c) at (5.1, 4.95) {\small $12600 [111110^{10}]$}; 
\foreach \from/\to in {b11/b12,b12/b13,b13/b14,b14/b15,b15/b16,b16/b17} \draw(\from)--(\to); 
\foreach \from/\to in {b21/b22,b22/b23,b23/b24,b24/b25,b25/b26} \draw(\from)--(\to); 
\foreach \from/\to in {b31/b32,b32/b33,b33/b34,b34/b35} \draw(\from)--(\to); 
\foreach \from/\to in {b41/b42,b42/b43,b43/b44} \draw(\from)--(\to); 

 \node[inner sep = 0.3mm] (b41e) at (9.4,-1.65) { $\,\,\,1$};
 \node[inner sep = 0.3mm] (b31e) at (8.4,-1.65) { $\,\,\,4$};
 \node[inner sep = 0.3mm] (b21e) at (7.4,-1.65) { $\,\,\,10$};
 \node[inner sep = 0.3mm] (b11e) at (6.4,-1.65) { $\,\,\,20$};
 \node[inner sep = 0.3mm] (b51e) at (12.5,-1.05) { $\,\,\,1$};
 \node[inner sep = 0.3mm] (b52e) at (12.5,-.05) { $\,\,\,5$};
 \node[inner sep = 0.3mm] (b53e) at (12.5,.95) { $\,\,\,15$};
 \node[inner sep = 0.3mm] (b54e) at (12.5,1.95) { $\,\,\,35$};
\node[inner sep = 0.3mm] (b55e) at (12.5,2.95) { $\,\,\,70$};
 \node[inner sep = 0.3mm] (b56e) at (12.5,3.95) { $\,\,\,126$};
 \node[inner sep = 0.3mm] (b57e) at (12.5,4.95) { $\,\,\,210$};

  \end{tikzpicture}
 \caption{The slice at $s_3 = 1$; the tallest vertical line is the slice at $s_4 = 5$. The binomial-coefficient factors for coordinates $s_4$ and $s_5$ are indicated in the margin.  The factor of $3$ from  coordinate $s_3$ is not shown.} 
 \label{fig:backslice}
\end{figure}
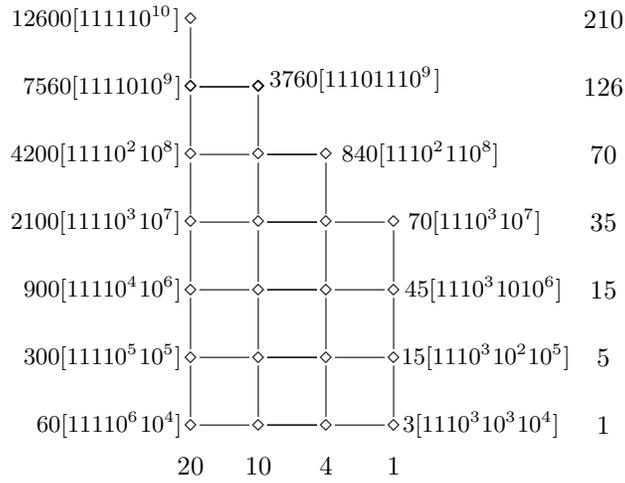
}

\showhide{
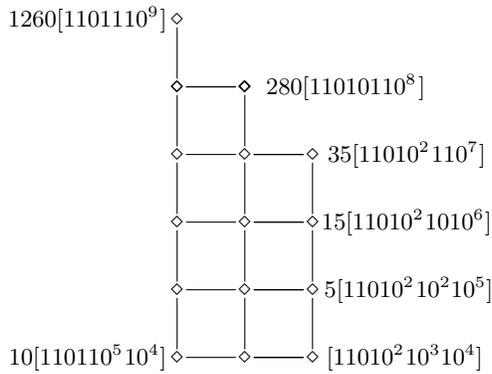
\begin{figure}
  \centering
  \begin{tikzpicture}[scale=.9]  

\node[inner sep = 0.3mm] (b12) at (6.5,-0.05) {\small $\diamond$}; \node[inner sep = 0.3mm] (b11c) at (5.3,-.05) {\small $10 [110110^510^4]\,\,\,\,$};

 \node[inner sep = 0.3mm] (b22) at (7.5,-0.05) {\small $\diamond$};
\node[inner sep = 0.3mm] (b32) at (8.5,-0.05) {\small $\diamond$};  \node[inner sep = 0.3mm] (b32c) at (9.8,-0.05) {\small $\,\, [11010^210^310^4]$};
\foreach \from/\to in {b12/b22,b22/b32,b22/b32} \draw(\from)--(\to);
\node[inner sep = 0.3mm] (b13) at (6.5, .95) {\small $\diamond$};
 \node[inner sep = 0.3mm] (b23) at (7.5, .95) {\small $\diamond$};
\node[inner sep = 0.3mm] (b33) at (8.5,  .95) {\small $\diamond$};  \node[inner sep = 0.3mm] (b32c) at (9.8, 0.95) {\small $\,\,\,\,5[11010^210^210^5]$};
\foreach \from/\to in {b13/b23,b23/b33,b23/b33} \draw(\from)--(\to);
\node[inner sep = 0.3mm] (b14) at (6.5, 1.95) {\small $\diamond$};
 \node[inner sep = 0.3mm] (b24) at (7.5, 1.95) {\small $\diamond$};
\node[inner sep = 0.3mm] (b34) at (8.5,  1.95) {\small $\diamond$};  \node[inner sep = 0.3mm] (b32c) at (9.8,1.95) {\small $\,\,\,\,15[11010^21010^6] \,$};
\foreach \from/\to in {b14/b24,b24/b34,b24/b34} \draw(\from)--(\to);
\node[inner sep = 0.3mm] (b15) at (6.5,  2.95) {\small $\diamond$};
 \node[inner sep = 0.3mm] (b25) at (7.5,  2.95) {\small $\diamond$};
\node[inner sep = 0.3mm] (b35) at (8.5,    2.95) {\small $\diamond$};  \node[inner sep = 0.3mm] (b32c) at (9.8,2.95) {\small $\,\,\,\,35[11010^2110^7] \,$};
\foreach \from/\to in {b15/b25,b25/b35,b25/b35} \draw(\from)--(\to);
\node[inner sep = 0.3mm] (b16) at (6.5,  3.95) {\small $\diamond$};
 \node[inner sep = 0.3mm] (b26) at (7.5,  3.95) {\small $\diamond$};
\node[inner sep = 0.3mm] (b36) at (7.5,  3.95) {\small $\diamond$};
\foreach \from/\to in {b16/b26,b26/b36,b26/b36} \draw(\from)--(\to);
\node[inner sep = 0.3mm] (b17) at (6.5,  3.95) {\small $\diamond$};\node[inner sep = 0.3mm] (b32c) at (8.9,3.95) {\small $\,\,\,\, 280 [11010110^8] \,$};
 \node[inner sep = 0.3mm] (b27) at (7.5,  3.95) {\small $\diamond$};
\foreach \from/\to in {b17/b27} \draw(\from)--(\to);
\node[inner sep = 0.3mm] (b17) at (6.5, 4.95) {\small $\diamond$}; \node[inner sep = 0.3mm] (b17c) at (5.3, 4.95) {\small $1260 [1101110^{9}]\,\,\,\,$}; 
\foreach \from/\to in {b12/b13,b13/b14,b14/b15,b15/b16,b16/b17} \draw(\from)--(\to); 

\foreach \from/\to in {b22/b23,b23/b24,b24/b25,b25/b26} \draw(\from)--(\to); 

\foreach \from/\to in {b32/b33,b33/b34,b34/b35} \draw(\from)--(\to); 

  \end{tikzpicture}
 \caption{The slice at $s_3 = 0$} 
 \label{fig:frontslice}
\end{figure}
}

\showhide{
\begin{figure}
  \centering
  \begin{tikzpicture}[scale=.9]  
%
 \node[inner sep = 0.3mm] (b21) at (7.95,-1.05) {\small $\diamond$};
\node[inner sep = 0.3mm] (b31) at (8.95,-1.05) {\small $\diamond$};\node[inner sep = 0.3mm] (b31c) at (9.5,-1.55) {\small $\,\,\,\,4[1101010^410^4]$};
\node[inner sep = 0.3mm] (b41) at (9.95,-1.05) {\small $\diamond$};  \node[inner sep = 0.3mm] (b41c) at (11.05,-.7) {\small $\,\,\,\,[11010^210^310^4]$};
\foreach \from/\to in {b21/b31,b31/b41} \draw(\from)--(\to);
\node[inner sep = 0.3mm] (b12) at (6.5,-0.05) {\small $\diamond$}; \node[inner sep = 0.3mm] (b41c) at (5.4,  0.2) {\small $60[11110^610^4]\,\,\,\,$};
 \node[inner sep = 0.3mm] (b22) at (7.5,-0.05) {\small $\diamond$};
\node[inner sep = 0.3mm] (b32) at (8.5,-0.05) {\small $\diamond$};
\node[inner sep = 0.3mm] (b42) at (9.5,-0.05) {\small $\diamond$}; \node[inner sep = 0.3mm] (b41c) at (10.7,  0.2) {\small $\,3[1110^310^310^4]$};
\foreach \from/\to in {b12/b22,b22/b32,b32/b42} \draw(\from)--(\to);

\foreach \from/\to in {b21/b22} \draw(\from)--(\to); 
\foreach \from/\to in {b31/b32} \draw(\from)--(\to); 
\foreach \from/\to in {b41/b42} \draw(\from)--(\to); 

 \node[inner sep = 0.3mm] (b21) at (7.5,1.05) {\small $\diamond$};\node[inner sep = 0.3mm] (b32c) at (9.4,2.25) {\small $840 [1110^2110^8] \,$};
\node[inner sep = 0.3mm] (b31) at (8.5,1.05) {\small $\diamond$};
\foreach \from/\to in {b21/b31,b21/b31} \draw(\from)--(\to);
\node[inner sep = 0.3mm] (b12) at (6,2.05) {\small $\diamond$}; \node[inner sep = 0.3mm] (b41c) at (9.9,  1.05) {\small $280 [11010110^8]$};
 \node[inner sep = 0.3mm] (b22) at (7,2.05) {\small $\diamond$};
\node[inner sep = 0.3mm] (b32) at (8,2.05) {\small $\diamond$}; \node[inner sep = 0.3mm] (b41c) at (4.5,  2.15) {\small $4200 [11110^210^8]$};

\foreach \from/\to in {b12/b22,b22/b32,b22/b32} \draw(\from)--(\to);

\foreach \from/\to in {b21/b22} \draw(\from)--(\to); 
\foreach \from/\to in {b31/b32} \draw(\from)--(\to); 
\foreach \from/\to in {b41/b42} \draw(\from)--(\to); 
  \end{tikzpicture}
 \caption{The slice at $s_5=4$ (top) and the slice at $s_5 = 0$ (bottom)} 
 \label{fig:bottom}
\end{figure}
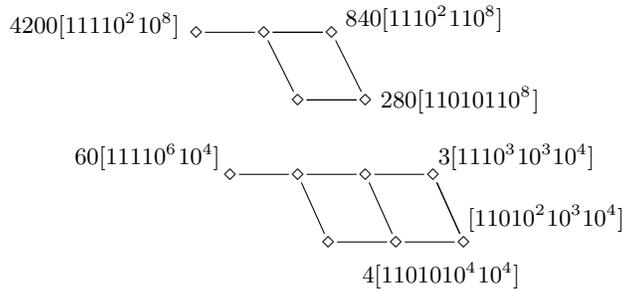
}

\showhide{
\begin{figure}
  \centering
  \begin{tikzpicture}[scale=.9]  
\node[inner sep = 0.3mm] (b11) at (6.5,-1.05) {\small $\diamond$};  \node[inner sep = 0.3mm] (b11c) at (6.5,-1.45) {\small $4 [1101010^410^4]\,$};
 \node[inner sep = 0.3mm] (b21) at (5.5,-.5) {\small $\diamond$};
\node[inner sep = 0.3mm] (b31) at (11,-1.05) {\small $\diamond$};  \node[inner sep = 0.3mm] (b32c) at (12.2,-1.45) {\small $[11010^210^310^4]$};
\node[inner sep = 0.3mm] (b41) at (10,-.5) {\small $\diamond$};  \node[inner sep = 0.3mm] (b41c) at (9.2,-1.0) {\small $\,\,\,\,3[1110^310^310^4]$};
\foreach \from/\to in {b11/b21,b31/b41} \draw(\from)--(\to);
\node[inner sep = 0.3mm] (b12) at (6.5,-0.05) {\small $\diamond$};
 \node[inner sep = 0.3mm] (b22) at (5.5,.5) {\small $\diamond$};
\node[inner sep = 0.3mm] (b32) at (11,-0.05) {\small $\diamond$};
\node[inner sep = 0.3mm] (b42) at (10,.5) {\small $\diamond$};
\foreach \from/\to in {b12/b22,b32/b42} \draw(\from)--(\to);
\node[inner sep = 0.3mm] (b13) at (6.5, .95) {\small $\diamond$};
 \node[inner sep = 0.3mm] (b23) at (5.5, 1.5) {\small $\diamond$};
\node[inner sep = 0.3mm] (b33) at (11,  .95) {\small $\diamond$};
\node[inner sep = 0.3mm] (b43) at (10, 1.5) {\small $\diamond$};
\foreach \from/\to in {b13/b23,b33/b43} \draw(\from)--(\to);
\node[inner sep = 0.3mm] (b14) at (6.5, 1.95) {\small $\diamond$};
 \node[inner sep = 0.3mm] (b24) at (5.5, 2.5) {\small $\diamond$};
\node[inner sep = 0.3mm] (b34) at (11,  1.95) {\small $\diamond$};
\node[inner sep = 0.3mm] (b44) at (10, 2.5) {\small $\diamond$};  \node[inner sep = 0.3mm] (b32c) at (10.45,2.85) {\small $\,\,\,\,105 [1110^31010^6]$};
\foreach \from/\to in {b14/b24,b34/b44} \draw(\from)--(\to);
\node[inner sep = 0.3mm] (b15) at (6.5,  2.95) {\small $\diamond$};\node[inner sep = 0.3mm] (b25c) at (7.4,3.4) {\small $\,\,\,\,280 [11010110^8]$};
\node[inner sep = 0.3mm] (b25) at (5.5,  3.5) {\small $\diamond$};\node[inner sep = 0.3mm] (b25c) at (5.10,3.9) {\small $\,\,\,\,840 [1110^2110^8]$};

\foreach \from/\to in {b15/b25} \draw(\from)--(\to);
\foreach \from/\to in {b11/b12,b12/b13,b13/b14,b14/b15} \draw(\from)--(\to); 
\foreach \from/\to in {b21/b22,b22/b23,b23/b24,b24/b25} \draw(\from)--(\to); 
\foreach \from/\to in {b31/b32,b32/b33,b33/b34} \draw(\from)--(\to); 
\foreach \from/\to in {b41/b42,b42/b43,b43/b44} \draw(\from)--(\to); 

  \end{tikzpicture}
 \caption{The slice at $s_4 = 1$ (left) and the slice at $s_4 = 0$ (right). Note that these slices are cartesian products. } 
 \label{fig:rightslices}
\end{figure}
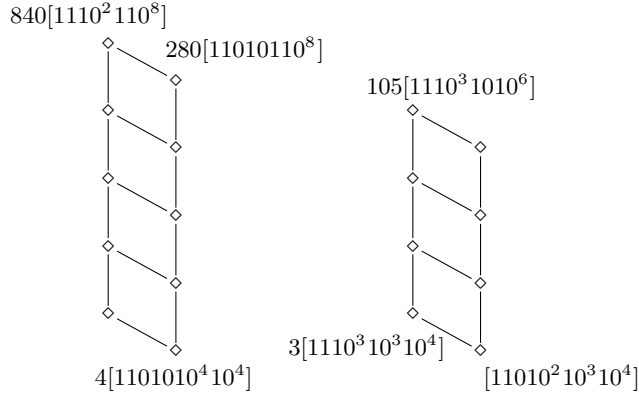
}

Applying the linear map $\mathsf{Sp}$ to the $\gamma$-basis elements, we obtain the polynomials $\bar\gamma(\underline{a};x,y)$ defined by  
\[
\bar\gamma(\underline{a};x,y) = \mathsf{Sp}(\gamma(\underline{a})).  
\]
The $\bar\gamma$-frame (or the \emph{gummy-bear frame}) for the vector space $\mathcal{T}(n,r)$ is the set $\{\bar\gamma(\underline{a};x,y): \underline{a} \in \mathcal{S}(n,r)\}$. 
The notion of a ``frame'' is borrowed from functional analysis.  
\footnote{Far more significant than just terminology are ideas.  One idea is filtering, that is, approximating a function with an expansion in terms of frame elements by removing those frame elements having ``small'' size.  Mortality permitting, we will pursue this, particularly for Tutte evaluations of the all-subset matroid.} 
Rather than having a basis, that is, a minimal spanning set, one decides to have more elements so that one is assured that every vector can be expanded as a linear combination having a desired property of frame elements.  The linear combination might no longer be unique.  For the $\bar\gamma$-frame, that desired property is given in the following theorem, an immediate consequence of Theorem \ref{Cat}.   

\begin{thm} \label{positivity}
The Tutte polynomial of an $(n,r)$-matroid $M$ can be written as a linear combination with \emph{non-negative integer coefficients}  of $\bar\gamma$-frame elements.  
One way to do this is 
\[
T(M;x,y) = \sum_{\underline{a}}  \nu(M;\underline{a}) \bar \gamma (\underline{a};x,y).
\]
\end{thm}

\section{The perfect matroid design analogy}\label{PMD}

Recall that an $(n,r)$ matroid $D$ is a \emph{perfect matroid design} or \emph{PMD} if for $0 \leq i \leq r,$ all the flats of rank $i$ have the same size $\xi_i,$ or equivalently,  all flags in $D$ have the same composition $\underline{a}$ (called its \emph{flag composition}), where $a_0 = \xi_0$ and $ a_i = \xi_i - \xi_{i-1}, \, 1 \leq i \leq r$.  Thus we have  
\[ 
\mathcal{G}(D) = \nu(D;\underline{a})   \gamma(\underline{a}), 
\]
where $\nu(D;\underline{a})$ is the number of flags in $D$.   
Specializing, we have 
\[ 
T(D;x,y) = \nu(D,\underline{a})   \bar \gamma(\underline{a};x,y).  
\eqno(5.1)\]

For this paper, the most important PMD is the uniform matroid $U_{r,n}$.  In $U_{r,n}$, all flags have composition $0,1,1, \ldots,1,n-r+1$.  Counting flags, we obtain, by Lemma \ref{uniform0}, 

\begin{align*}
 \bar\gamma (0,1,1,\ldots,1,n-r+1;x,y) 
= \tfrac {(n-r+1)!}{n!}  \big( \tau(r,n-r;x) + \tau(n-r,r;y) \big)
\end{align*} 

Other PMDs whose Tutte polynomial can be easily calculated are direct sums of $r$ copies of the rank-$1$ matroid $U_{1,m}$ with $m$ parallel rank-$1$ elements and $l$ loops.  Their Tutte polynomials give explicit formulas for rank-$1$ $\bar\gamma$-frame elements.    

\begin{lemma} \label{easyevaluation}
$
\bar\gamma(l,m,m, \ldots,m;x,y) = \tfrac {1}{r!} y^l (x + y + y^2 + \cdots + y^{m-1})^r.  
$
\end{lemma} 

\begin{dfn} \label{parameters} Let $\underline{a} = a_0,a_1, \ldots,a_r$.   We define 
\begin{align*}
& \nu(\underline{a}) = \frac {(a_1+a_2+a_3+\cdots+a_r)(a_2 +a_3+\cdots+a_r)\cdots (a_{r-1}+a_r)}{a_1a_2a_3 \cdots a_{r-1}}, 
\\
& f_k(\underline{a}) =  \frac {1}{ \nu(0,a_1,a_2, \ldots,a_k) \nu (0,a_{k+1},a_{k+2}, \ldots,a_r)}.
\end{align*}
\end{dfn}
Note that $\nu(\underline{a})$ is defined so that $a_0$ is (intentionally) ignored; thus, $\nu(a_0,a_1,\ldots,a_r)=\nu(0,a_1,\ldots,a_r)$.  Note also that 
\begin{align*}
& \nu(\underline{a})  
= \frac {(\xi_r-\xi_0)(\xi_r-\xi_1)(\xi_r-\xi_2) \cdots (\xi_r-\xi_{r-2}) }{a_1a_2a_3 \cdots a_{r-1} }
\end{align*}
and 
\begin{align*}
 \nu(\underline{a})f_k(\underline{a}) 
& = \frac {\nu(0,a_1, \ldots, a_r)}{ \nu(0,a_1,a_2, \ldots,a_k) \nu (0,a_{k+1},a_{k+2}, \ldots,a_r)}
\\
& = 
\frac {(\xi_r-\xi_0)(\xi_r-\xi_1)(\xi_r-\xi_2) \cdots (\xi_r-\xi_{k-1})}{(\xi_k-\xi_0) (\xi_k - \xi_1) (\xi_k -\xi_2) \cdots (\xi_k - \xi_{k-1})}.
\end{align*} 
The parameter $\nu(\underline{a})$ gives the number of flags and the number $\nu(\underline{a})f_k(\underline{a})$ gives the number of flats of rank $k$ (all of which have size $\xi_k$) if 
there exists a PMD $D$ with flag composition $\underline{a}$.    Loops in $D$ do not affect the parameters $\nu(\underline{a})$ and $f_k(\underline{a})$.    If there is no PMD with flag composition 
$\underline{a}$, then $\nu(\underline{a})$ or $\nu(\underline{a})f_k (\underline{a})$ are usually {\it not} integers, but from their definition, they are always positive rational numbers.  

Equation (5.1) suggests the \emph{perfect matroid design analogy:}  {\it $\nu (\underline{a}) \bar\gamma (\underline{a};x,y)$ behaves formally like the Tutte polynomial of a PMD. }
For example, adding a loop preserves the property of being a PMD and results in a multiplicative factor of $y$ to the Tutte polynomial.  The following lemma says this is true for $\bar\gamma$-frame elements as well.   This lemma allows us to assume $a_0=0$ without any loss of generality. 

\begin{lemma} \label{loops} 
$
\bar\gamma (a_0,a_1, a_2, \ldots, a_r;x,y) = y^{a_0} \bar\gamma (0,a_1, a_2, \ldots, a_r;x,y).
$
\end{lemma} 

\begin{proof}  
We will prove the lemma in the form: if $a_0 \geq 1$, then 
\[
\bar\gamma (a_0,a_1, a_2, \ldots, a_r;x,y) = y  \bar\gamma (a_0-1,a_1, a_2, \ldots, a_r;x,y).
\]
In an $(n,r)$-matroid $M$ containing a flag with composition $a_0,a_1,\ldots,a_r$ with $a_0 \geq 1$, choose such a flag $Z_0,Z_1, \ldots,Z_r$, and a loop in $Z_0$. Relabel the elements  of $M$ so that the chosen loop is labeled $n$.   Then as the loop $n$ can be inserted or ``shuffled'' into any position in a permutation $\pi$ on $\{1,2,\ldots,n-1\}$ having flag 
$Z_0 \backslash n,Z_1, \ldots,Z_r$ in the deletion $M \backslash n,$    
$\gamma(a_0,a_1,\ldots,a_r)$ can be written as a sum of smaller sums $\sigma(\pi)$ of $n+1$ symbols,  where  
\begin{align*} 
& \sigma(\pi)  = [\underline{r} ( n, \pi(1), \pi(2), \ldots, \pi(n-1))] +[\underline{r}(\pi(1), n, \pi(2), \ldots, \pi(n-1))] 
\\
& \qquad\qquad +
[\underline{r}(\pi(1), \pi(2), n, \ldots, \pi(n-1))] + \cdots + 
[\underline{r}(\pi(1), \pi(2), \ldots, \pi(n-1),n)].
\end{align*} 
For example,  if $\underline{r}(\pi) = 110010$, then 
\[
\sigma(\pi) = [0110010] + [1010010] + 3[1100010] + 2[1100100].
\]
Let $\sigma (\pi;x,y) = \mathsf{Sp}(\sigma (\pi)) $, the Tutte-polynomial specialization of $\sigma(\pi). $ We will now show that the following identity holds, from which the lemma follows.  
Then 
\[
\sigma(\pi;x,y) = n \big(1 + (y-1) \big) [\underline{r}(\pi)].
\] 
The proof is an easy book-keeping exercise.  An example should make it clear what one needs to do.    
Let $\underline{r}(\pi) = 110010$.  Note that $110010$ corresponds to the composition $0,1,3,2$. Then using the notation $X = x-1$ and $Y =y-1$, 
\begin{align*} 
7 (1 + Y)   &   [110010;x,y] = 7X^3 + 42X^2 + 105 X + 140 XY + 105 XY^2  + 42Y^2 + 7Y^3
\\
& \qquad + 7X^3Y + 42X^2Y + 105 XY + 140 XY^2 + 105 XY^3  + 42Y^3 + 7Y^4.
\end{align*} 
The Tutte-polynomial specializations of the $7$ symbols occurring in $\sigma(\pi)$ are  
\begin{align*}
[0110010;x,y] &= X^3 + 7X^3Y + 21X^2Y + 35XY + 35XY^2 + 21XY^3 + 7Y^3 + Y^4
\\
[1010010;x,y] &= X^3 + 7X^2 + 21X^2Y  + 35XY + 35 XY^2  + 21XY^3 + 7Y^3+ Y^4
\\
3[1100010;x,y] &= 3X^3 + 21X^2  + 63X + 105 XY + 105XY^2  +  63XY^3 + 21Y^3 + 3Y^4
\\
2[1100100;x,y] &= 2 X^3 + 14X^2 + 42X + 70 XY + 70XY^2 + 42Y^2 + 14 Y^3 + 2Y^4
\end{align*} 
and so $\sigma(\pi;x,y)$ is the sum of the polynomials on the right-hand side.   Summing by columns from left to right, we obtain 
\begin{align*}
\sigma(\pi;x,y) &= 7X^3 + 7X^3Y + 42 X^2Y + 35XY + 35XY^2 + 21XY^3 + 7Y^3 + Y^4
\\
\,                  & \qquad + 42 X^2 + 105 X  + 35XY + 35 XY^2  + 21XY^3 + 7Y^3+ Y^4
\end{align*} 
Finally, To show $\sigma(\pi;x,y) = 7(1+Y) [\underline{r}(\pi)]$, we match monomials in the two expressions, going down each column, from left to right.  
\end{proof}

The analog of Lemma \ref{loops} is false  for isthmuses: adding an isthmus to a PMD does not yield a PMD with the exception of $U_{n,n}$.  Examples can easily be found.

The next lemma shows that $\nu(\underline{a}) \bar\gamma(\underline{a};1,1)$ can be calculated as if it were the number of bases of a PMD.  

\begin{lemma} Let $\underline{a}$ be an $(n,r)$-composition with $a_0 =0$.   Then 
\[
\bar \gamma(\underline{a};1,1) = \frac {a_1a_2 \cdots a_r}{r!}
\]
and
\[
\nu (\underline{a}) \bar\gamma (\underline{a};1,1) = \frac {\xi_n(\xi_n-\xi_1)(\xi_n-\xi_2) \cdots (\xi_n-\xi_{r-1})}{r!}.  
\]
\end{lemma}  
\begin{proof}   It suffices to prove the first formula.   Let $\underline{b}$ be an $(n,r)$-sequence.  Then it is easy to see that the specialization $[\underline{b};1,1]$ equals $0$ with the only exception being 
$\underline{b} = 1^r0^{n-r}$.  In that case, 
\[
[1^r0^{n-r}; 1,1] = \frac {1}{n!} \binom {n}{r} = \frac {1} {r! \, (n-r)!}. 
\]
The sequence $1^r0^{n-r}$ is obtained from $\underline{a}$ by shifting every $1$-bit as far as possible to the left.  Thus, from Definition \ref{defgamma}, the coefficient of $[1^r0^{n-r}]$ is 
\[
a_1! a_2! \cdots a_r! \prod_{i=2}^r  \binom {a_1 +  \cdots + a_{i-1} +  a_i - i}{a_1 + \cdots + a_{i-1} - i + 1}.
\]
This binomial-coefficient product ``telescopes'' and after further cancellations, the entire product simplifies to 
\[
a_1a_2a_3 \cdots a_r (n-r)! \,.
\]
Multiplying this by $[1^r0^{n-r};1,1]$ completes the proof.  
\end{proof} 

\section{Norms of slices} \label{slicenorm}

In this section, we give a formula for a $\bar\gamma$-frame element involving norms of slices.  
Let $A$ be a subset of $\mathsf{Box}(\underline{a})$.  The \emph{norm} $\|A\|$ of $A$ is defined by 
\[
\|A\| = \sum_{(s_1,s_2,\ldots,s_r) \in A}  \mathsf{c}_{\underline{a}} (s_1,s_2,\ldots,s_r) ,
\]
where $\mathsf{c}_{\underline{a}}$ is the coefficient function (see Definition \ref{defgamma}).  
 Norms are valuations on subsets of $\mathsf{Box}(\underline{a})$;  in particular,  they are additive for disjoint unions (hence, subtractive for set differences) and multiplicative on cartesian products.   

\begin{lemma}
$  \underline{a}!  \| [\underline{a})\| = \gamma(\underline{a}; [\underline{b}] \to 1), $
where the right hand side is the integer obtained from $\gamma (\underline{a})$ by specializing every symbol $[\underline{b}]$ to the integer $1$, in other words, the number of symbols occurring in 
$\gamma(\underline{a})$, counted with multiplicity.  
\end{lemma} 

We can now state the combinatorial form of the formula for $\bar\gamma(\underline{a};x,y)$.  
It can be viewed as an application of Lemma \ref{syzygyiterated}.

\begin{formula}\label{formula0}   Let $\underline{a}$ be an $(n,r)$-composition with $a_0=0$.  
Then 
\begin{align*}
&\bar\gamma(\underline{a};x,y) = \frac {\underline{a}! \, \| [\underline{a})\| }{n!}  T(U_{r,n};x,y)   + 
\\
& (xy-x-y) \sum_{k=1} ^{r-1} (x-1)^{r-k-1} \left(\sum_{j= 0}^{\xi_{k}-k-1}   \frac {\underline{a}! \,\|[\underline{a}; s_{k+1} \leq j)\|}{(\xi_k - j)!(n-\xi_k+j)!}  (y-1)^{\xi_{k} - k - j-1} \right).
\end{align*} 
\end{formula} 

\begin{proof}  We move every $1$-bit in every specialized symbol in the expression for  $\bar\gamma(\underline{a};x,y)$ as far to the left as possible.  This yields $\underline{a}! \, \|[\underline{a}) \|$
copies of $[1^r0^{n-r};x,y]$.  For a given rank $k$ and position $m$, we group together the syzygy terms for all the moves of the bit $1_{k+1}$ from position $m+1$ to $m$;  the possible positions are when 
$k < m \leq \xi_k$ and at position $\xi_k - j + 1$, precisely $\underline{a}! \, \|[\underline{a};s_{k+1} \leq j)\|$ $1_{k+1}$-bits are moved, those that are at position $\xi_k - j + 1$ and those that are to the right and have to pass through that position.  
\end{proof} 
\noindent

\showhide{
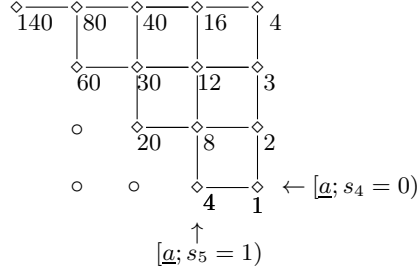
\begin{figure}
  \centering
  \begin{tikzpicture}[scale=0.8]  
\node[inner sep = 0.3mm] (b31) at (3.5,-1.05) {\small $\diamond$};  \node[inner sep = 0.3mm] (b31x) at (3.7,-1.3) {\small $4$};
\node[inner sep = 0.3mm] (b41) at (4.5,-1.05) {\small $\diamond$};  \node[inner sep = 0.3mm] (b41x) at (4.5,-1.35) {\small $1$};
\node[inner sep = 0.3mm] (b31a) at (3.5,-1.8) {\small $\uparrow$};  \node[inner sep = 0.3mm] (b31y) at (3.7,-2.15) {\small $[\underline{a};s_5=1)$};
\node[inner sep = 0.3mm] (b41a) at (5.1,-1.06) {\small $\leftarrow$};  \node[inner sep = 0.3mm] (b41y) at (6.3,-1.05) {\small $[\underline{a};s_4=0)$};
\foreach \from/\to in {b31/b41} \draw(\from)--(\to);
 \node[inner sep = 0.3mm] (s21) at (2.45,-1.05)  {$\circ$};  \node[inner sep = 0.3mm] (b31x) at (3.7,-1.3) {\small $4$};
\node[inner sep = 0.3mm] (s11) at (1.5,-1.05) {$\circ$};  \node[inner sep = 0.3mm] (b41x) at (4.5,-1.35) {\small $1$};
\node[inner sep = 0.3mm] (s12) at (1.5,-.1) {$\circ$};  \node[inner sep = 0.3mm] (b31x) at (3.7,-1.3) {\small $4$};
%
 \node[inner sep = 0.3mm] (b22) at (2.5,-0.05) {\small $\diamond$};\node[inner sep = 0.3mm] (b21x) at (2.7,-.3) {\small $20$};
\node[inner sep = 0.3mm] (b32) at (3.5,-0.05) {\small $\diamond$};\node[inner sep = 0.3mm] (b32x) at (3.7,-.3) {\small $8$};
\node[inner sep = 0.3mm] (b42) at (4.5,-0.05) {\small $\diamond$};\node[inner sep = 0.3mm] (b42x) at (4.7,-.3) {\small $2$};
\foreach \from/\to in {b22/b32,b22/b32,b32/b42} \draw(\from)--(\to);
\node[inner sep = 0.3mm] (b13) at (1.5, .95) {\small $\diamond$};\node[inner sep = 0.3mm] (b13x) at (1.7,.7) {\small $60$};
 \node[inner sep = 0.3mm] (b23) at (2.5, .95) {\small $\diamond$};\node[inner sep = 0.3mm] (b23x) at (2.7,.7) {\small $30$};
\node[inner sep = 0.3mm] (b33) at (3.5,  .95) {\small $\diamond$};\node[inner sep = 0.3mm] (b33x) at (3.7,.7) {\small $12$};
\node[inner sep = 0.3mm] (b43) at (4.5, .95) {\small $\diamond$};\node[inner sep = 0.3mm] (b43x) at (4.7,.7) {\small $3$};
\foreach \from/\to in {b13/b23,b23/b33,b23/b33,b33/b43} \draw(\from)--(\to);
\node[inner sep = 0.3mm] (b04) at (0.5, 1.95) {\small $\diamond$};\node[inner sep = 0.3mm] (b04x) at (.8,1.7) {\small $140$};
\node[inner sep = 0.3mm] (b14) at (1.5, 1.95) {\small $\diamond$};\node[inner sep = 0.3mm] (b14x) at (1.8,1.7) {\small $80$};
 \node[inner sep = 0.3mm] (b24) at (2.5, 1.95) {\small $\diamond$};\node[inner sep = 0.3mm] (b24x) at (2.8,1.7) {\small $40$};
\node[inner sep = 0.3mm] (b34) at (3.5,  1.95) {\small $\diamond$};\node[inner sep = 0.3mm] (b34x) at (3.8,1.7) {\small $16$};
\node[inner sep = 0.3mm] (b44) at (4.5, 1.95) {\small $\diamond$};\node[inner sep = 0.3mm] (b44x) at (4.8,1.7) {\small $4$};
\foreach \from/\to in {b04/b14,b14/b24,b24/b34,b24/b34,b34/b44} \draw(\from)--(\to);
\foreach \from/\to in {b13/b14,b22/b23,b23/b24,b33/b43,b41/b42,b42/b43,b43/b44,b31/b32,b32/b33,b33/b34} \draw(\from)--(\to);
\end{tikzpicture}
 \caption{The filter $[0,1,1,4,2,4)$ with ``shadow''  vectors} 
 \label{fig:example011424}
\end{figure}
}

A example might help us understand Formula \ref{formula0}.
Let $\underline{a}$ be the $(12,5)$-composition $0,1,1,4,2,4$. From Figure \ref{fig:example011424}, we can find the numbers need in Formula \ref{formula0};  for example,   
\[
\|[\underline{a})\| = 420, \,\,\, \|[\underline{a}; s_5 = 1)\| = 40, \,\,\, \|[\underline{a}; s_5 \leq 2)\| = 140,  \,\,\, \|[\underline{a}: s_4 = 0)\| = 5.
\] 
From these numbers, we obtain, 
\[
\frac{\underline{a}! \|[\underline{a})\| }{n!} = \frac {4!2!4! \cdot 420}{12!}  = \frac {4\cdot 2}{12\cdot 11\cdot 10 \cdot 6} = \frac {1}{990} = \frac {1}{\nu(0,1,1,4,2,4)},
\]
$f_4(\underline{a}) = \frac {1}{84}, \, f_{4,3} (\underline{a})= \frac {2}{90},$ and $f_3 (\underline{a}) = \frac {1}{90}$.  
Formula \ref{formula0} says that $\bar\gamma(\underline{a})$ equals 
\begin{align*}
&  \tfrac {1}{990}T(U_{12,5}) 
+ 4!2!4! (xy-x-y)\left( \tfrac {10}{8!4!} (y-1)^3+  \tfrac {50}{7!5!} (y-1)^2+\tfrac {150}{6!6!} (y-1)+\tfrac {350}{5!7!}\right) 
\\
& + 4!2!4! (xy-x-y)(x-1)\left( \tfrac {5}{6!6!} (y-1)^2+  \tfrac {35}{5!7!} (y-1) + \tfrac {140}{4!8!}\right),
\end{align*}
To anticipate the next two sections, we will ``approximate'' the slices $[\underline{a}:s_5=2)$ and $[\underline{a}:s_5=3)$ by including and excluding the ``shadow'' vectors, 
and rewrite the polynomial at rank $4$ as 
\[
\left(\tfrac {10}{8!4!} (y-1)^3+  \tfrac {50}{7!5!} (y-1)^2+\tfrac {160}{6!6!} (y-1)+\tfrac {420}{5!7!}\right) - \left(\tfrac {10}{6!6!} (y-1)+\tfrac {70}{5!7!} \right),
\]
so that 
\begin{align*}
& \qquad \bar\gamma(0,1,1,4,2,4;x,y) = \tfrac {1}{990}T(U_{12,5}) 
\\
& + (xy-x-y)\left(\tfrac {1}{84} ( (y-1)^3+  8(y-1)^2+ 28 (y-1)+56) + \tfrac {2}{90}((y-1)+ 6) )\right) 
\\
& + (xy-x-y)(x-1)\left( \tfrac {1}{90} ((y-1)^2+  6(y-1) + 15)\right).
\end{align*}

We will convert the combinatorially described norms in Formula \ref{formula0} to algebraic expressions.  To do this, we need several technical counting results for slices of filters.  
 For a slice obtained by imposing the equation $s_{j}=i$ (and possibly other equations),  every coefficient in its norm is multiplied by the binomial coefficient 
\[
\binom {a_{j}+i-1}{i}
\]
coming from the $s_{j}$ coordinate.  We call this factor the \emph{binomial factor} contributed by $s_{j}$.  

\begin{lemma} \label{easycase}  Let $0 \leq i \leq a_1-1$.  Then 
\[
\| [0,a_1,a_2, \ldots,a_r; s_2 =i) \| = \binom {a_2+i-1}{i} \| [0,a_2+i,a_3,\ldots,a_r) \|
\]
\end{lemma} 
\begin{proof} 
The map sending $(0,i,s_3, \ldots,s_r)$ to $(s_3,\ldots,s_r)$ is a bijection from $[0,a_1,a_2, \ldots,a_r; s_2 =i)$ 
to $[0,a_2+i,a_3,\ldots,a_r)$.  
This bijection preserves norms, once we put in the binomial factor contributed by $s_2$. 
\end{proof} 

Norms of filters satisfy a simple recursion. 
\begin{lemma}\label{normrecurse} 
\[
\| [0,a_{1},a_{2},a_3,  \ldots, a_r)\|  = \sum_{i=0}^{a_1 -1} \binom {a_2 + i- 1}{i}  \| [0,a_{2}+i, a_{3},  \ldots, a_r)\|
\]
\end{lemma}
\begin{proof} 
A filter (indeed, any subset of $\mathsf{Box}(\underline{a})$) is a disjoint union of its basic slices: in particular, 
\[
[0,a_1,\ldots,a_r) = \bigcup_{i=0}^{a_1-1}  [0,a_1,\ldots,a_r;s_2 = i).
\]
Taking norms, we obtain the equation in the lemma.  
\end{proof}

\begin{thm} \label{numberofsymbols} 
\[
\|[0,a_1,a_2, \ldots, a_r)\|  = \frac {(a_1+a_2+\cdots+a_r)!}{ 
\big(\prod_{i=1}^r (a_i+a_{i+1} + \cdots+a_r)\big)\big(\prod_{i=1}^r (a_i-1)! \big) }.
\]
\end{thm} 

Theorem \ref{numberofsymbols} follows by telescoping the binomial coefficients in the product in the next lemma.  

\begin{lemma} \label{numberofsymbols1} 
\begin{align*} 
\| [0,a_1,a_2, \ldots,a_r)\|
 = \prod_{i=1}^{r-1} \binom {a_r+a_{r-1}+ \cdots+a_i-1}{a_i - 1} .
\end{align*} 
\end{lemma} 
\begin{proof}  
We will prove the lemma by induction on $r$, the number of nonzero terms in $0,a_1,\ldots,a_r$.  When $r=1$ and there is exactly one bit sequence $10^{a_1-1}$ in $[0,a_1)$ and    
hence, $\| [0,a_1)\|$ equals $1$, the value of the empty product, in agreement with the lemma.
Next, we apply the induction hypothesis to obtain 
\[
\| [0,a_{2}+i, a_{3},  \ldots, a_r)\| =  \binom {a_r + a_{r-1} + \cdots + a_{2} + i - 1}{a_{2}+i-1} Q, 
\] 
where 
\[
Q = \prod_{i=3}^{r-1}  \binom {a_r+a_{r-1} + \cdots+a_{i} -1}{a_{i}-1}.
\]
By Lemma \ref{normrecurse},
\begin{align*} 
\|[0,a_{1},a_{2},\ldots,a_r)\|  =   Q \sum_{i=0}^{a_1-1} \binom {a_{2} +i - 1}{i}  \binom {a_r + a_{r-1} + \cdots + a_{2} + i - 1}{a_{2}+i-1} ,
\end{align*}    
and the proof can now be completed by the second identity in Identities \ref{binomialID1} with $A = a_{r:3}$, $a= a_2$, and $\beta = a_1 - 1$.
\end{proof} 

We note two corollaries.  The first states the inductive hypothesis in another way and is recorded for easy reference later.  

\begin{cor} \label{relativesymbols}
\begin{align*} 
\| [0,a_1,a_2, \ldots,a_r)\|
 = \|[0,a_k,a_{k+1}, \ldots,a_r)\| \prod_{i=1}^{k-1} \binom {a_r+ a_{r-1}+ \cdots + a_{i}-1}{a_{i}-1} .
\end{align*} 
\end{cor}

The second corollary is obvious either from Lemma \ref{numberofsymbols1} or directly.  It is useful shortcut when doing hand calculations.  

\begin{cor} \label{reomveones} $\| [0,1^{\iota-1}, a_{\iota},a_{\iota+1}, \ldots, a_r) \| = \| [0,a_\iota,a_{\iota+1}, \ldots, a_r) \|$.
\end{cor}

A minor adjustment of the argument in the proof of Lemma \ref{numberofsymbols1} takes care of the case when there are loops.
\begin{lemma}\label{withloops} 
$
\| [a_0,a_1,a_2,  \ldots, a_r) \|  = \binom {a_0+a_1+ \cdots+a_r}{a_0} \| [0,a_1,a_2, \ldots, a_r) \|.
$
\end{lemma}

The last two lemmas in this section are kitten steps towards a formula for the norms of slices obtained by fixing $s_{k+1}$.   

\begin{lemma} \label{independence}  For $0 \leq j \leq a_{k:1}  - k$, 
\begin{align*}
& \| [0,a_1, \ldots,a_r; s_{k+1} = j) \| = \| [0,a_1, \ldots,a_{k+1}; s_{k+1} = j)\| \, \|[0,a_{k+1}+j, a_{k+2}, \ldots,a_r) \| .
\end{align*} 
\end{lemma} 
\begin{proof} 
A shift vector $(0,s_2,\ldots,s_k,j,s_{k+2},\ldots,s_r)$ in $[0,a_1, \ldots,a_r; s_{k+1} = j)$ can be split into two vectors $(0,s_2,\ldots,s_k,j)$ and $(0,s_{k+2},\ldots,s_r)$, where the first vector is 
in $[0,a_1, \ldots,a_{k+1}; s_{k+1} = j)$ and the second vector is in $[0,a_{k+1}+j, a_{k+2},   \ldots, a_r)$.  This split gives a bijection from $[0,a_1, \ldots,a_r; s_{k+1} = j)$ to the cartesian product 
\[
[0,a_1, \ldots,a_{k+1}; s_{k+1} = j) \times [0,a_{k+1}+j, a_{k+2}, \ldots,a_r).
\]
\end{proof} 

The binomial factor contributed by $s_{k+1}$ belongs to the first component of the cartesian product; however, we can extract it and put it into the norm of the second component to obtain a simplification.  

\begin{lemma} \label{independence1}
\[ 
\binom {a_{k+1} + j -1}{j} \|[0,a_{k+1}+j, \ldots,a_r)\| = \binom {a_{r:k+1} + j - 1}{j} \, \|[0, a_{k+1}, \ldots,a_r)\| .
\]
\end{lemma} 

\begin{proof}  By Lemma \ref{relativesymbols}, the left-hand side equals 
\[
\binom {a_{k+1} + j -1}{j}  \binom {a_{r:k+1} + j - 1}{a_{k+1}+j-1} \, \|[0, a_{k+2},\ldots,a_r)\| .
\]
The pair of binomial coefficients can be manipulated using Lemma \ref{binomialID1} to obtain 
\[
\binom {a_{r:k+1}-1}{a_{k+1}-1}  \binom {a_{r:k+1} + j - 1}{j} \, \|[0, a_{k+2},\ldots,a_r)\| .
\]
By Lemma \ref{relativesymbols}, 
\[
\binom {a_{r:k+1}-1}{a_{k+1}-1} \|[0, a_{k+2},\ldots,a_r)\| = \|[0, a_{k+1},\ldots,a_r)\|.
\]
Using this, we arrive at the right-hand side.
\end{proof}

We shall derive an alternating-sum formula for 
  the norm $ \| [0,a_1,\ldots,a_{r};s_{k+1}=j) \|$.   This formula may not be the most efficient if our aim is to obtain a counting formula, so this is best done in the next section in the context of the formula we wish to derive. 
  
\section{A formula for $\bar\gamma$-frame elements}   \label{formula}

In this section, we convert the combinatorially described norms in Formula \ref{formula0} into algebraic expressions.  To do so, we need one more set of parameters.
\begin{dfn}
The \emph{rank-$k$ leading coefficient} $f_{k,k}(\underline{a})$ is defined to be $f_k(\underline{a})$ and 
for $1 \leq h \leq k-1$, the \emph{rank-$k$ interior coefficients} $f_{k,k-h}(\underline{a})$ are defined by 
\[
f_{k,k-h}(\underline{a}) = m_{k,k-h} (\underline{a}) f_{k-h}(\underline{a}). 
\]
where 
\[ 
m_{k,k-h} (\underline{a})
= \frac {a_{r:k+1} a_{r:k} a_{r:k-1} \cdots a_{r:k-h+2} } {a_{k:k-h+1} a_{k-1;k-h+1}a_{k-2;k-h+1}  \cdots a_{k-h+1:k-h+1} }. 
\]
\end{dfn} 
\noindent 
The factor $m_{k,k-h} (\underline{a})$ is called the \emph{M\"obius multiplier}.  For example, 
\begin{align*} 
& m_{k,k-1} (\underline{a}) = \frac {a_{r:k+1} }{a_{k:k}} = \frac {a_r+a_{r-1}+\cdots+a_{k+1}} {a_k} 
\\
& m_{k,k-2} (\underline{a}) = \frac {a_{r:k+1} a_{r:k}}{a_{k:k-1} a_{k-1:k-1} } = \frac {a_{r:k+1}(a_{r:k+1}+a_k)} {(a_k + a_{k-1})a_{k-1} } 
\\
& m_{k,k-3} (\underline{a}) = \frac {a_{r:k+1} a_{r:k} a_{r:k-1}}{a_{k:k-2} a_{k-1:k-2}a_{k-2:k-2}} = \frac {a_{r:k+1}(a_{r:k+1}+a_k)(a_{r:k+1} + a_k + a_{k-1})} 
 {(a_k + a_{k-1}+a_{k-2})(a_{k-1} + a_{k-2})a_{k-2}  } 
 \\
& m_{k,k-4} (\underline{a}) = \frac {a_{r:k+1}(a_{r:k+1}+a_k)(a_{r:k+1} + a_k + a_{k-1})(a_{r:k+1} + a_k + a_{k-1}+a_{k-2})} 
 {(a_k + a_{k-1}+a_{k-2}+a_{k-3})(a_{k-1} + a_{k-2}+a_{k-3})(a_{k-2} + a_{k-3})  a_{k-3} } \, .
\end{align*} 

\begin{formula}\label{gummybear} 
\begin{align*}
& \bar\gamma(0,a_1,a_2, \ldots, a_r;x,y) = f_r(\underline{a}) T(U_{r,n};x,y) 
\\
& \qquad + 
(xy-x-y)\sum_{k=1}^{r-1}  (x-1)^{r-k-1}  \left( \sum_{h=0}^{k-1} (-1)^h f_{k,k-h}(\underline{a}) \tau(k+1, \xi_{k-h}  - k  - 1;y)\right).
\end{align*}
\end{formula}
\noindent 
If we multiply Formula \ref{gummybear} on both sides by $\nu(0,a_1,\ldots,a_r)$, we obtain a formula for the Tutte polynomial of a PMD with flag composition $0,a_1,\ldots,a_r$ (if one exists).   

We will derive Formula \ref{gummybear} from Formula \ref{formula0}.  By Theorem \ref{numberofsymbols}, 
\[
\frac {\underline{a}! \|[\underline{a})\|}{n!} =  \frac { a_1a_2 \cdots a_{r-1}}{a_{1:r} a_{2:r} \cdots a_{r-1:r} } = \frac {1}{\nu(0,a_1,\ldots,a_r)}. 
\]
Thus, the coefficient of $T(U_{r,n})$ is as stated.  The hard work is to show that for $1 \leq k \leq r-1$, the polynomial $R_k(y)$, defined by 
\[
R_k(y) = \sum_{j= 0}^{\xi_{k}-k-1}   \frac {\underline{a}! \,\|[\underline{a}; s_{k+1} \leq j)\|}{(\xi_k - j)!(n-\xi_k+j)!}  (y-1)^{\xi_{k} - k - j-1},
\]
equals the polynomial 
\[ 
\sum_{h=0}^{k-1} (-1)^h f_{k,k-h}(\underline{a}) \tau(k+1, \xi_k -k - 1;y).
\]
 
 We begin with the leading coefficient, that is, the coefficient of $y^{\xi_k - k -1}$ in $R_k(y)$.  In this case, $j=0$ and by Theorem \ref{numberofsymbols} and Lemma \ref{independence}, 

 \begin{align*}
& \frac {\underline{a}! \| [ 0,a_1, a_2, \ldots, a_r;s_{k+1}=0)\|}{\xi_k!(n-\xi_k)! } 
\\
&= \frac {a_1!a_2!\cdots a_k!\| [ 0,a_1, \ldots,a_{k}) \|} {\xi_k!} \cdot \frac {a_{k+1}!\cdots a_r! \|[0,a_{k+1}, \ldots, a_r)\| }{(n - \xi_k)! } 
\\
&=  \frac {a_1a_2\cdots a_{k-1}}{a_{1:k} a_{2:k}\cdots a_{k-1:k} } \cdot \frac {a_{k+1}a_{k+2} \cdots a_{r-1} }{a_{k+1:r} a_{k+2:r}
\cdots a_{r-1:r}} 
\\
&=
\frac {1}{\nu(0,a_1, \ldots,a_{k})\nu(0,a_{k+1},a_{k+2}, \ldots,a_r)} 
\\
&= f_{k,k}(\underline{a}).
\end{align*} 

To proceed further, we need to delve into the geometry of the filter $[\underline{a})$.    By Lemma \ref{stayright}, the bit $1_i$ must stay to the right of the bit $1_{i-1}$; hence, a vector 
$(s_1,s_2,\ldots,s_r)$  in 
 $\mathsf{Box}(\underline{a})$ is the shift vector of a bit sequence in $[\underline{a})$ if and only if it satisfies the following inequalities:   
$s_1 \leq a_0$,  and for $2 \leq i \leq r$,
\[
 s_i \leq  a_{i-1} - 1 + s_{i-1}.  
\eqno(\mathrm{I}_i) \]
Put another way, the filter $[\underline{a})$ is the ``discrete polyhedron'' defined by the linear inequalities  $s_1 \leq a_0$ and $\mathrm{I}_i$, $2 \leq i \leq r$, in $\mathsf{Box}(\underline{a})$. 
\footnote{The inequalities $\mathrm{I}_i$ are quite simple and so norms of slices can be obtained by modifying standard methods for finding the number of points inside a polyhedron.  We will do the counting differently because we want to have a formula with coefficients related to the number of flats.  Formulas based on other counting strategies may well be useful and worth exploring.}

 We shall need to consider other polyhedra.   The polyhedron $\mathsf{P}_{h}(0,a_1,\ldots,a_{k+1})$ is defined to be the set of vectors in $\mathsf{Box}(0,a_1,\ldots,a_k,a_{k+1})$ 
satisfying the inequalities $\mathrm{I}_2, \ldots, \mathrm{I}_{k-h}$, but not necessarily $\mathrm{I}_{k-h+1}, \ldots, \mathrm{I}_{k+1}$.   
Slices of the polyhedron $\mathsf{P}_h(0,a_1,\ldots,a_{k+1})$ are defined in the same way as for a filter.  We shall always work with polyhedra defined with the composition $0,a_1,\ldots,a_{k+1}$ and we will not specify the composition, so that, for example, we write $\mathsf{P}_h(s_{k+1}=j)$ instead of 
$\mathsf{P}_h(0,a_1,\ldots,a_{k+1}; s_{k+1}=j)$.  

Inequalities $\mathrm{I}_{2}, \ldots, \mathrm{I}_k$ remain in force in $\mathsf{P}_0$ and so 
the slice $\mathsf{P}_0 (s_{k+1}=j)$ consists of vectors $(0,s_2,\ldots,s_k,j)$  (not necessarily shift vectors) such that 
$(0,s_2,\ldots,s_k)$ is a shift vector in $[0,a_1,\ldots,a_k)$.  Thus, in a intuitive and informal way, we may consider $\mathsf{P}_0(s_{k+1}=j)$  a zeroth-order approximation to 
$[0,a_1,\ldots,a_{k+1}; s_{k+1}=j)$ and write 
\[
\mathsf{P}_0(s_{k+1}=j) \approx_{\, 0}  [0,a_1,\ldots,a_{k+1}; s_{k+1}=j).
\]
Taking a product and then a union, we have 
\[
\mathsf{P}_0(s_{k+1}=j) \times  [0,a_{k+1}+j,\ldots,a_r) 
\approx_{\, 0} 
[0,a_1,\ldots,a_{r}; s_{k+1}=j)  
\]
and 
\[
\bigcup_{\alpha=0}^j  \mathsf{P}_0(s_{k+1} = j) \times  [0,a_{k+1}+j,\ldots,a_r) 
\approx_{\, 0} 
[0,a_1,\ldots,a_{r}; s_{k+1} \leq j) \,. 
\]
Taking this one step further, we define the polynomial $P_{k,0}(y)$  by 
\[
P_{k,0}(y) = \sum_{j=0}^{\xi_k-k-1} \frac { \underline{a}! \| \bigcup_{\alpha=0}^j  \mathsf{P}_0(s_{k+1} = j) \times  [0,a_{k+1}+j,\ldots,a_r) \|}{(\xi_k-j)!(n-\xi_k+j)!}   (y-1)^{\xi_k-k-j-1} .
\]
Then, 
\[
R_k(y) \approx_{\, 0} P_{k,0}(y).  
\]
All three approximations are exact in the range $0 \leq j \leq a_k-1$ because inequality $\mathrm{I}_{k+1}$ is satisfied there for any non-negative value of $s_k$.  
Since 
\[
\| \mathsf{P}_0 (s_{k+1}=j) \| = \binom {a_{k+1}+j-1}{j} \| [0,a_1,\ldots,a_k) \| , 
\]
we can use Lemma \ref{independence1} to obtain 
\[
\| \mathsf{P}_0 (s_{k+1}=j) \| \|[0,a_{k+1}+j,\ldots,a_r)\|= \binom {n-\xi_k+j-1}{j} \|[0,a_1,\ldots,a_k)\| \|[0,a_{k+1},\ldots,a_r)\|
\]
and hence, 
\[
\| \bigcup_{\alpha=0}^j  \mathsf{P}_0(s_{k+1} = j) \times  [0,a_{k+1}+j,\ldots,a_r) \| = \binom {n-\xi_k+j}{j} \| [0,a_1,\ldots,a_k) \| \|[0,a_{k+1},\ldots,a_r)\|.
\] 
The coefficient of $(y-1)^{\xi_i-k-j-1}$ in $P_{k,0}(y)$ can now be calculated: 
\begin{align*} 
& \frac { \underline{a}! \| \bigcup_{\alpha=0}^j  \mathsf{P}_0(s_{k+1} = j) \times  [0,a_{k+1}+j,\ldots,a_r) \|}{(\xi_k-j)! (n-\xi_k+j)!} 
\\
& = \binom {n-\xi_k+ j}{j} \frac {\underline{a}! \|[0,a_1, \ldots,a_k)\|\| [ 0,a_{k+1},\ldots,a_r)\|}{(\xi_i-j)!(n-\xi_k+j)! }
\\
& =  \frac {(n- \xi_k +j)!} {j! (n-\xi_k)!}  \cdot  \frac {\xi_k (\xi_k-1)\cdots (\xi_k-j+1)}{\xi_k!} \cdot \frac  {\underline{a}! \| [ 0,a_{1},\ldots,a_k)\|\|[0,a_{k+1},\ldots,a_r)\|}{(n-\xi_k+j)! }    
\\
& = f_{k.k}(\underline{a}) \binom {\xi_k}{j}.
\end{align*}
We conclude that 
\[
P_{k,0}(y) = f_{k,k}(\underline{a})\left(  \binom{\xi_k}{\xi_k-k-1} + \cdots + \binom {\xi_k }{1}(y-1)^{\xi_k-k-2}+ (y-1)^{\xi_k-k-1}\right).
 \]
By Identity \ref{tauinadifferentform} and the fact that $(\xi_k - k -1) + (k+1) = \xi_k$, 
\[ 
P_{k,0}(y) = f_{k,k} (\underline{a}) \tau(k+1, \xi_k-k-1;y).
\] 
Because the set approximations are exact when $0 \leq j \leq a_k-1$,  the approximation $R_k(y) \approx_{\,0} P_{k,0}(y)$ is exact for degrees higher than $\xi_{k-1} - k -2$.  We state this formally in the next lemma.  

\begin{lemma}\label{exact0}  The difference $R_k(y) - f_{k,k} (\underline{a}) \tau(k+1,\xi_k-k-1;y)$ has degree $\xi_{k-1}-k-2$.   
\end{lemma} 

When $j \geq a_k$, we put in more terms to get better approximations. 
For the index $j$ where the coordinate $s_{k+1}$ is fixed, we define  \emph{auxiliary indices}  $j_h$ by $j_0=j$, and for $h \geq 1$, 
\[
j = a_k + a_{k-1} + \cdots + a_{k-h-1} + j_h = a_{k:k-h-1} + j_h
\] 
when the equation determines a non-negative integer $j_h$; otherwise, $j_h$ is not defined (and irrelevant).   The auxiliary index $j_h$ tracks how much $j$ exceeds $a_{k:k-h-1}$.  
For example, if $j \geq a_k + a_{k-1}$, then $j_1 = j - a_k = j_2+ a_{k-1}$ and $j_2 = j - (a_k + a_{k-1})$.   

When $a_{k} \leq j_0 \leq a_k + a_{k-1} -1$, there are vectors in $\mathsf{P}_{0}(s_{k+1} = j_0)$ that do not satisfy inequality $\mathrm{I}_{k+1}$.  Loosely speaking, they are the vectors for which $s_k$ is too small.  Indeed, setting $s_{k+1}=j_0$ in $\mathrm{I}_{k+1}$ yields 
 \[
 s_k \geq j_0-a_k+1 = j_1+1.
 \]  
Thus, the vectors in $\mathsf{P}_{0}(s_{k+1}=j_0)$ that do {\it not} satisfy $\mathrm{I}_{k+1}$ are exactly those vectors for which $s_k \leq j_1$.  Thus, 
a first-order approximation to $[0,a_1,\ldots,a_{k+1}: s_{k+1} =j_0)$, holding over a larger range of $j_0$, is the set-difference 
\[ 
\mathsf{P}_{0}(s_{k+1}=j_0) -  \mathsf{P}_{1}(s_{k+1}=j_0, s_k \leq j_1) .
\]
Note that the set-difference is constructed so that inequality $\mathrm{I}_{k+1}$ holds for all non-negative values of $j_0$.
Taking products, we have 
\begin{align*} 
& [0,a_1,\ldots,a_r: s_{k+1}= j_0) 
\\
&\qquad \approx_{\, 1}  
\big( \mathsf{P}_{0}(s_{k+1}=j_0) -  \mathsf{P}_{1}(s_{k+1}=j_0, s_k \leq j_1) \big) \times [0,a_{k+1}+ j_0,a_{k+2}, \ldots,a_r) \,
\end{align*} 
and 
\begin{align*}
& [0,a_1,\ldots,a_r: s_{k+1} \leq j_0)  \approx_{\, 1} 
\mathsf{P}_{0}(s_{k+1}\leq j_0)  \times [0,a_{k+1}+ j_0,a_{k+2}, \ldots,a_r)
\\
&\qquad\qquad  -  \bigcup_{\alpha=0}^{j_1}   \mathsf{P}_{1}(s_{k+1}=a_k+\alpha, s_k \leq \alpha)  \times [0,a_{k+1}+ a_k + \alpha,a_{k+2}, \ldots,a_r) .
\end{align*}
These approximations are exact when $0 \leq j \leq a_{k:k-1}-1$.  
Let $P_{k,1}(y)$ be the polynomial 
\begin{align*} 
&  \sum_{j_1=0}^{\xi_{k-1}-k-1} \frac {(y-1)^{\xi_{k-1} -k - j_1 -1} }{(\xi_{k-1}-j_1)! (n-\xi_{k-1} + j_1)!}
\\
& \qquad\qquad\qquad  \underline{a}! \left\| \bigcup_{\alpha=0}^{j_1}   \mathsf{P}_1(s_{k+1} =a_k+ \alpha , s_k \leq \alpha) \times [0,a_{k+1}+a_k+\alpha , a_{k+2}, \ldots,a_r) \right\|
\end{align*} 
To simplify $P_{k,1}(y)$, we perform the following calculations.  First, omitting the binomial factor contributed by $s_{k+1}$, we have 
\begin{align*}
\| \mathsf{P}_{1}(s_{k+1}=a_k+\alpha , s_k \leq \alpha) \|  & = \sum_{\beta=0}^{\alpha} \| \mathsf{P}_{1}(s_{k+1}= a_k+\alpha, s_k =\beta) \|
\\
& = \sum_{\beta=0}^{\alpha}  \binom {a_k + \beta - 1}{\beta} \|[0,a_1,\ldots,a_{k-1}) \| 
\\
& =  \binom {a_k + \alpha}{\alpha} \|[0,a_1,\ldots,a_{k-1}) \| .
\end{align*} 
Thus, using Lemma \ref{independence1}, 
\begin{align*} 
& \underline{a}! \left\| \bigcup_{\alpha=0}^{j_1}  \mathsf{P}_1(s_{k+1} =a_k+ \alpha , s_k \leq \alpha) \times [0,a_{k+1}+a_k+\alpha , a_{k+2}, \ldots,a_r) \right\|
\\
& \qquad = \sum_{\alpha=0}^{j_1} \underline{a}!  \| [0,a_1,\ldots,a_{k-1})\|\binom {a_k+\alpha}{\alpha}  \binom {a_{r:k+1} +a_k+ \alpha-1}{a_k+\alpha} \|[0,a_{k+1},\ldots , a_r) \| 
\\
& \qquad = \sum_{\alpha=0}^{j_1} \underline{a}! \| [0,a_1,\ldots,a_{k-1})\| \binom {a_{r:k}-1}{a_k} \binom {n-\xi_{k-1} + \alpha -1}{\alpha}   \|[0,a_{k+1},\ldots,a_r) \| 
\\
& \qquad = \binom {a_{r:k}-1}{a_k} \binom {n-\xi_{k-1} + j_1}{j_1} \underline{a}! \| [0,a_1,\ldots,a_{k-1})\|   \|[0,a_{k+1},\ldots,a_r) \| \, 
\end{align*} 
Setting $j_1=0$, the leading coefficient of $P_{k,1}(y)$ is 
\[
\binom {a_{r:k}-1}{a_k} \frac {\underline{a}! \| [0,a_1,\ldots,a_{k-1})\|   \|[0,a_{k+1},\ldots,a_r) \|} {\xi_{k-1} ! (n-\xi_{k-1})!} .
\]
Let us assume, for the time being, that the coefficient we calculated is indeed $f_{k,k-1}(\underline{a})$ as defined.  
The coefficient of $(y-1)^{\xi_{k-1} - k-j_1-1}$ in $P_{k,1}(y)$ can now be simplified: 
\begin{align*} 
&\binom {a_{r:k}-1}{a_k} \binom {n-\xi_{k-1} + j_1}{j_1} \frac {\underline{a}! \| [0,a_1,\ldots,a_{k-1})\|   \|[0,a_{k+1},\ldots,a_r) \| }{(\xi_{k-1}-j_1)! (n-\xi_{k-1} + j_1)!}
\\
&\qquad =\binom {a_{r:k}-1}{a_k} \binom {\xi_{k-1}}{j_1} \frac {\underline{a}! \| [0,a_1,\ldots,a_{k-1})\|   \|[0,a_{k+1},\ldots,a_r) \| }{\xi_{k-1}! (n-\xi_{k-1})!}
\\
& \qquad = \binom {\xi_{k-1}}{j_1} f_{k,k-1}(\underline{a}).  
\end{align*}  
We conclude that 
\[
P_{k,1}(y) = f_{k,k-1}(\underline{a}) \left(  \binom{\xi_{k-1}}{\xi_{k-1} -k-1} + \cdots + \binom {\xi_{k-1} }{1}(y-1)^{\xi_{k-1}-k-2}+ (y-1)^{\xi_{k-1} -k-1}\right),
\]
that is, 
\[
P_{k,1}(y) = f_{k,k-1} (\underline{a}) \tau(k+1,\xi_{k-1}-k-1;y).
\] 
Summarizing, we have obtained a first-order approximation for $R_k(y)$.  We state this formally in the following lemma.  

\begin{lemma}\label{exact1}  The alternating sum 
\[
R_k(y) - f_{k,k} (\underline{a}) \tau(k+1, \xi_k-k-1;y) + f_{k,k-1} (\underline{a}) \tau(k+1,\xi_{k-1}-k-1;y)
\]
has degree $\xi_{k-2}-k-2$.
\end{lemma} 

\showhide{
\begin{figure}
  \centering
  \begin{tikzpicture}[scale=1]  
\node[inner sep = 0.0mm] (a11) at (-1.8, 2.45) {}; \node[inner sep = 0.3mm] (aa) at (-.65, 1.95) {$=$}; 
\node[inner sep = 0.0mm] (a12) at (-1.25, 2.45) {};
\node[inner sep = 0.0mm] (a22) at (-1.25, 1.85) {};  
\foreach \from/\to in {a11/a12,a11/a22,a12/a22} \draw(\from)--(\to);
\node[inner sep = 0.0mm] (c11) at (1.3, 1) {};\node[inner sep = 0.3mm] (cc) at (2.15, 1.75) {$-$}; 
\node[inner sep = 0.0mm] (c12) at (1.8, 1) {};
\node[inner sep = 0.0mm] (c21) at (0, 2.45) {}; 
\node[inner sep = 0.0mm] (c22) at (1.8, 2.45){}; 
\foreach \from/\to in {c11/c12,c11/c21,c12/c22,c21/c22} \draw(\from)--(\to);
\node[inner sep = 0.0mm] (e11) at (2.5, 1) {};\node[inner sep = 0.3mm] (cc) at (4.1, 1.75) {$+$}; 
\node[inner sep = 0.0mm] (e12) at (3.8, 1) {};
\node[inner sep = 0.0mm] (e21) at (2.5, 2.45) {}; 
\node[inner sep = 0.0mm] (e22) at (3.8, 2.45){}; 
\foreach \from/\to in {e11/e12,e11/e21,e12/e22,e21/e22} \draw(\from)--(\to);
\node[inner sep = 0.0mm] (g21) at (4.5, 1.8) {};\node[inner sep = 0.0mm] (g22) at (5.3, 1) {}; 
\node[inner sep = 0.mm] (g11) at (4.5, 1) {};
\foreach \from/\to in {g11/g21,g21/g22,g11/g22} \draw(\from)--(\to);
  \end{tikzpicture}
 \caption{Inclusion-exclusion visualized} 
 \label{fig:inexclusion}
\end{figure}
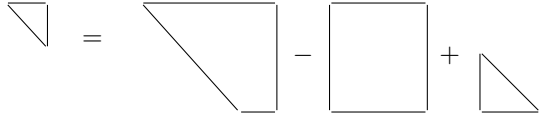}

The next case is when $a_{k:k-1} \leq j_0 \leq a_{k:k-2}-1$. Here, there are two inequalities in play.   
The vectors that do not satisfy $\mathrm{I}_k$ are those that satisfy $s_k \leq j_1$ and they were removed in the first-order approximation.   However, inequality $\mathrm{I}_k$ remains in force in $\mathsf{P}_{0}(s_{k+1}=j_0)$ so that not all the vectors in $\mathsf{P}_1(s_{k+1}=j_0, s_k \leq j_1)$ are in $\mathsf{P}_{0}(s_{k+1}=j_0)$.  A simple calculation shows that when $s_{k+1}=j_0$,  inequality 
$\mathrm{I}_{k-1}$ is not satisfied if and only if  for $0 \leq m \leq j_2$,  $s_k \leq a_{k-1} + m$, $s_{k-1} \leq m$. 
Thus, in the range $a_{k:k-1} \leq j_0 \leq a_{k:k-2}-1$, the slice $[0,a_1, \ldots, a_{k+1}: s_{k+1} =j_0)$ equals   
\begin{align*}
& \mathsf{P}_{0}(s_{k+1}=j_0) 
\\
& \qquad - \,  \left(  \mathsf{P}_{1}(s_{k+1}=j_0, s_k \leq j_1)
\,-\,  \bigcup_{\beta=0}^{j_2} \mathsf{P}_{2}(s_{k+1}=j_0, s_k = a_{k-1}+\beta, s_{k-1} \leq \beta) \right) .
\end{align*}
This set equation is illustrated in Figure \ref{fig:inexclusion}.   
We will use this as the second-order approximation to $[0,a_1, \ldots, a_{k+1}: s_{k+1} =j_0)$.  Noting that $j_0=a_k+a_{k-1}+j_2 =a_{k:k-1} + j_2$, the second term in the approximation to 
$[0,a_1, \ldots, a_r: s_{k+1} =j_0)$ is
\[
 \bigcup_{\beta=0}^{j_2} \mathsf{P}_{2}(s_{k+1}=a_{k:k-1}+j_2, s_k = a_{k-1}+\beta, s_{k-1} \leq \beta)  \times [0,a_{k+1} + a_{k:k-1} + j_2,\ldots,a_r)
\]
and the second term in the approximation to $[0,a_1,\ldots,a_r;s_{k+1} \leq j_0)$ is
\[
\bigcup_{\alpha=0}^{j_2} \bigcup_{\beta=0}^{\alpha} \mathsf{P}_{2}(s_{k+1}=a_{k:k-1} + \alpha, s_k = a_{k-1}+\beta, s_{k-1} \leq \beta) \times [0,a_{k+1}+a_{k:k-1} + \alpha,\ldots,a_r).
\]
We calculate its norm, starting inside the double union.  Omitting the binomial factor from $s_{k+1}$, we have 
\begin{align*} 
& \left\| \mathsf{P}_{2}(s_{k+1}=a_{k:k-1}+\alpha, s_k = a_{k-1}+\beta, s_{k-1} \leq \beta) \right\|
\\
& = \sum_{\gamma=0}^\beta  \| \mathsf{P}_{2}(s_{k+1}=a_{k:k-1}+\alpha, s_k = a_{k-1}+\beta, s_{k-1} = \gamma) \|
\\
& = \|[0,a_1,\ldots,a_{k-2})\| \binom {a_k + a_{k-1} + \beta - 1}{a_{k-1}+\beta} \left( \sum_{\gamma=0}^{\beta} \binom {a_{k-1} +\gamma-1} {\gamma} \right) 
\\
& =\|[0,a_1,\ldots,a_{k-2})\| \binom {a_k + a_{k-1} + \beta - 1}{a_{k-1}+ \beta} \binom {a_{k-1} +\beta }{\beta} 
\\
& =\|[0,a_1,\ldots,a_{k-2})\| \binom {a_k + a_{k-1} - 1}{a_{k-1}} \binom {a_k +a_{k-1} +\beta -1}{\beta}  .
\end{align*} 
Summing over $\beta$ and omitting the binomial factor from $s_{k+1}$, we obtain
\begin{align*} 
& \left\| \bigcup_{\beta=0}^{\alpha} \mathsf{P}_{2}(s_{k+1}=a_{k:k-1}+\alpha, s_k = a_{k-1}+\beta, s_{k-1} \leq \beta) \right\|
\\
& = \sum_{\beta=0}^{\alpha} \|[0,a_1,\ldots,a_{k-2})\| \binom {a_k + a_{k-1} - 1}{a_{k-1}} \binom {a_k +a_{k-1} +\beta -1}{\beta} 
\\
& = \|[0,a_1,\ldots,a_{k-2})\| \binom {a_k + a_{k-1} - 1}{a_{k-1}} \binom {a_k +a_{k-1} + \alpha}{\alpha} \, .
\end{align*} 
We now sum over $\alpha$,  with the binomial factor from $s_{k+1}$.  Using Lemma \ref{independence1}, we have  
\begin{align*} 
&\left\|\bigcup_{\alpha=0}^{j_2} \bigcup_{\beta=0}^{\alpha} \mathsf{P}_{2}(s_{k+1}=a_{k:k-1} + \alpha, s_k = a_{k-1}+\beta, s_{k-1} \leq \beta) \times [0,a_{k+1}+a_{k:k-1} + \alpha,\ldots,a_r) \right\|
\\
& =
 \sum_{\alpha=0}^{j_2} \left\| \bigcup_{\beta=0}^{\alpha} \mathsf{P}_{2}(s_{k+1}=a_{k:k-1}+\alpha, s_k = a_{k-1}+\beta, s_{k-1} \leq \beta) \right\| \| [0,a_{k+1}+a_{k:k-1} + \alpha,\ldots,a_r) \|
\\
&=  \sum_{\alpha=0}^{j_2} \|[0,a_1,\ldots,a_{k-2})\| \binom {a_k+a_{k-1}-1}{a_{k-1}} \binom  {a_k + a_{k-1}+\alpha}{\alpha} 
\\
& \qquad\qquad\qquad\qquad\qquad\qquad  \binom {a_{k+1}+ a_{k:k-1} + \alpha-1} {a_{k:k-1} + \alpha} \| [0,a_{k}+a_{k:k-1} + \alpha,\ldots,a_r) \|
\\
&=  \sum_{\alpha=0}^{j_2} \|[0,a_1,\ldots,a_{k-2})\| \binom {a_k+a_{k-1}-1}{a_{k-1}} \binom  {a_k + a_{k-1}+\alpha}{\alpha}  \binom {a_{r:k-1}+ \alpha-1} {a_{k:k-1} + \alpha} \| [0,a_{k+1},\ldots,a_r) \|
\\
&=  \sum_{\alpha=0}^{j_2} \|[0,a_1,\ldots,a_{k-2})\| \binom {a_k+a_{k-1}-1}{a_{k-1}}  \binom  {a_{r:k-1}-1}{a_k + a_{k-1}} \binom {a_{r:k-1} + \alpha -1}{\alpha} \| [0,a_{k+1},\ldots,a_r) \| .
\end{align*} 

Let $P_{k,2}(y)$ be the polynomial 
\begin{align*} 
&  \sum_{j_2=0}^{\xi_{k-2}-k-1} \frac { (y-1)^{\xi_{k-2} -k - j_2 -1} }{(\xi_{k-1}-j_2)! (n-\xi_{k-1} + j_2)!}
\\
&  \cdot  \sum_{\alpha=0}^{j_2}  \underline{a}!
\left\| \bigcup_{\beta=0}^{\alpha} \mathsf{P}_2(s_{k+1}=a_{k:k-1}+\alpha, s_{k} =a_{k-1} + \beta , s_{k-1} \leq \beta) \times [0,a_{k+1} + a_{k:k-1}+\alpha, \ldots,a_r) \right\| .
\end{align*} 
This is the second term in the second-order approximation to $R_k(y)$. 
Its leading coefficient can be obtained by setting  $j_2 =0$.  There is then only one term in the sum and union and the leading coefficient 
equals 
\begin{align*} 
 \|[0,a_1,\ldots,a_{k-2})\| \binom {a_k + a_{k-1} - 1}{a_{k-1}} \binom {a_{r:k-1}-1} {a_k+a_{k-1}} \|[0,a_{k+1},\ldots,a_r)\|.
\end{align*} 
Once again, we assume this coefficient equals $f_{k,k-2}(\underline{a})$ as defined earlier, leaving the verification for later.     
As in the earlier cases, we obtain 
\[
P_{k,2}(y) = f_{k,k-2} (\underline{a}) \tau(k+1,\xi_{k-2} - k - 1;y) 
\]
and a second-order approximation which is exact in the sense given in the next lemma.  

\begin{lemma}\label{exact1}  The alternating sum 
\begin{align*}
& R_k(y) - f_{k,k} (\underline{a}) \tau(k+1, \xi_k-k-1;y) + f_{k,k-1} (\underline{a}) \tau(k+1,\xi_{k-1}-k-1;y)
\\
& \qquad\qquad\qquad - f_{k,k-2} (\underline{a}) \tau(k+1,\xi_{k-2}-k-1;y)
\end{align*}
has degree $\xi_{k-3}-k-2$.
\end{lemma} 

It should be clear now how to proceed in general.  We include and exclude alternately, getting better approximations until at the $k-1$ step, the degree of the alternating sum is zero, and we obtain Formula \ref{gummybear}.   The missing step is to show that the interior coefficients equal  $f_{k,k-h}(\underline{a})$ as defined.   We will do the case $h=4$.  This is sufficiently large to be typical.  
The fourth term in the fourth-order approximation for $[0,a_1,\ldots,a_{k+1} ;s_{k+1} = a_{k:k-3}+ j_4)$ is the union 
\[
\bigcup_{\alpha=0}^{j_4} \bigcup_{\beta=0}^{\alpha} \bigcup_{\gamma=0}^{\beta} \bigcup_{\delta=0}^{\gamma} \bigcup_{\epsilon=0}^{\delta} 
\, \mathsf{P}_4(\alpha,\beta,\gamma,\delta,\epsilon) ,
\]
where $\mathsf{P}_4(\alpha,\beta,\gamma,\delta,\epsilon)$ is the slice in $\mathsf{P}_4$ defined by the following equations :  

\begin{align*} 
& s_{k+1} = a_k + a_{k-1} + a_{k-2}+a_{k-3} + \alpha,  
\\
& \qquad   \left\{ \begin{array}{cccc}
s_{k} = a_{k-1}+a_{k-2} + a_{k-3},  s_{k} = a_{k-1}+a_{k-2} + a_{k-3} + 1,  \ldots  \qquad\qquad
\\
s_{k} = a_{k-1}+a_{k-2} + a_{k-3} + \beta,   
\ldots, s_{k} = a_{k-1}+a_{k-2} + a_{k-3} + \alpha,  
\end{array} \right. 
\\
& \qquad\qquad \left\{ \begin{array}{cccc}
s_{k-1} = a_{k-2} + a_{k-3},  
s_{k-1} = a_{k-2} + a_{k-3} + 1,  \ldots   \qquad\qquad\quad
\\
s_{k-1} = a_{k-2} + a_{k-3} + \gamma,   
\ldots , 
s_{k-1} = a_{k-2} + a_{k-3} + \beta,  
\end{array} \right. 
\\
& \qquad\qquad\qquad   \left\{ \begin{array}{cc}
s_{k-2} = a_{k-3},  
s_{k-2} = a_{k-3} + 1,  
\ldots  \qquad\qquad\qquad
\\  s_{k-2} = a_{k-3} + \delta,  \ldots ,
s_{k-2} = a_{k-3} + \gamma,  
\end{array} \right. 
\\
& \qquad\qquad\qquad\qquad   \left\{ \begin{array}{cc}
s_{k-3} = 0,  
s_{k-3} = 1,  
\ldots, s_{k-3} = \epsilon,  
\ldots ,
s_{k-3} = \delta,  
\end{array} \right. 
\end{align*}
The norm $\| \mathsf{P}_4(\alpha,\beta,\gamma,\delta,\epsilon) \|$ has the following form
\[
\|[0,a_1,\ldots,a_{k-4})\| \Theta  \binom {a_{r:k-3} + \alpha-1}{a_{k:k-3} + \alpha} \|[0,a_{k+1},\ldots,a_r) \|,
\]
where $\Theta$ is an expression in binomial coefficients.  We will calculate $\Theta$, working our way up. 
For the lowest row (where the equations fixing $s_{k-3}$ are located), $\Theta$ is calculated as follows:  
\begin{align*}
&\binom {a_{k-2}+a_{k-3}+\delta-1}{a_{k-3} + \delta}  \sum_{\epsilon=0}^{\delta}  \binom {a_{k-3}+\epsilon-1}{\epsilon}
=
\binom {a_{k-2}+a_{k-3}+\delta-1}{a_{k-3} + \delta}  \binom {a_{k-3}+\delta}{\delta}
\\
&\qquad\qquad\qquad\qquad  = \binom {a_{k-2}+a_{k-3}-1}{a_{k-3}} \binom {a_{k-2}+a_{k-3} +\delta -1}{\delta} .
\end{align*} 
Going up to the next row, we have
\begin{align*}
&\binom {a_{k-2}+a_{k-3}-1}{a_{k-3}}\binom {a_{k-1}+ a_{k-2}+a_{k-3}+\gamma-1}{a_{k-2}+ a_{k-3} + \gamma}  \sum_{\delta=0}^{\gamma}  \binom {a_{k-2}+ a_{k-3}+\delta-1}{\delta}
\\
& \qquad =
\binom {a_{k-2}+a_{k-3}-1}{a_{k-3}}\binom {a_{k-1}+a_{k-2}+a_{k-3}+\gamma-1}{a_{k-2}+a_{k-3} + \gamma}  \binom {a_{k-2}+ a_{k-3}+\gamma}{\gamma}
\\
&  \qquad  = \binom {a_{k-2}+a_{k-3} -1}{a_{k-3}} \binom {a_{k-1}+a_{k-2}+a_{k-3}-1}{a_{k-2} +a_{k-3}} \binom {a_{k-2}+a_{k-3} +\gamma -1}{\gamma} .
\end{align*} 
In a similar way,  the next row give 
\begin{align*} 
& \binom {a_{k-2}+a_{k-3} -1}{a_{k-3}} \binom {a_{k-1}+a_{k-2}+a_{k-3}-1}{a_{k-2} +a_{k-3}}  \binom {a_k + a_{k-1}+a_{k-2}+a_{k-3}-1}{a_{k-1}+ a_{k-2} +a_{k-3}} 
\\
& \qquad\qquad\qquad\qquad\qquad\qquad\qquad\qquad \binom {a_{k-1}+a_{k-2}+a_{k-3} +\beta -1}{\beta} .
\end{align*} 

The top row requires a different treatment as we will calculate the norm rather than $\Theta$.  Putting in the binomial factor from  $s_{k+1}$, as modified by Lemma \ref{independence1}, the norm is $\|[0,a_1,\ldots,a_{k-4})\|\|[0,a_{k+1},\ldots,a_r)\|$ times 
\begin{align*} 
\binom {a_{k-2}+a_{k-3} +\alpha -1}{\alpha-1} \binom {a_{r:k-3} + \alpha-1}{a_{k:k-3} + \alpha}\prod_{i=0}^2 \binom {a_{k-i:k-3}-1}{a_{k-i-1:k-3}}  .
\end{align*}
Messing about with binomial coefficients, we arrive at the final answer:  $\| \mathsf{P}_4(\alpha,\beta,\gamma,\delta,\epsilon)\|$ equals $\|[0,a_1,\ldots,a_{k-4})\|\|[0,a_{k+1},\ldots,a_r)\|$ times 

\begin{align*} 
\binom {a_{r:k-3} -1}{a_k + a_{k-1}+a_{k-2} + a_{k-3}} \binom {a_{r:k-3} + \alpha-1}{\alpha}\prod_{i=0}^2 \binom {a_{k-i:k-3}-1}{a_{k-i-1:k-3}}  .
\end{align*}

Setting $\alpha=0$ and thinking of $4$ as $h$, we obtain the following formula.   

\begin{formula} \label{intcoefficients} 
The coefficient $f_{k,k-h} (\underline{a})$ equals 
 \begin{align*}  
\binom {a_{r:k-h+1}-1}{a_{k:k-h+1}} \prod_{i=0}^{h-2} \binom {a_{k-i:k-h+1}-1}{a_{k-i-1:k-h+1}}  \frac {\underline{a}!  \|[0,a_1,\ldots,a_{k-h})\| \|[0,a_{k+1},\ldots,a_r)\| }
{\xi_{k-h}! (n - \xi_{k-h})! }
\end{align*}
\end{formula} 
\noindent 
Formula \ref{intcoefficients} is not obviously the same as the expression for $f_{k,k-h}(\underline{a})$ as defined just before Formula \ref{gummybear}.  To
 reconcile the two, recall, from Corollary \ref{relativesymbols} that  

\begin{align*}
& \| [0,a_1,a_2, \ldots,a_{k-h})\| \|[0,a_{k-h+1}, a_{k-h+2}, \ldots, a_r)\|
\\
& \qquad  = \| [0,a_1,a_2, \ldots,a_{k-h})\| \prod_{i=k-h+1}^{k} \binom {a_{r:i}-1}{a_i-1}   \|[0,a_{k+1},\ldots,a_r)\| 
\end{align*}
Comparing this expression with the one in Formula \ref{intcoefficients}, we see that $m_{k,k-h} (\underline{a})$ is the correct multiplier. 
With the profusion of indices, the products of binomial coefficients are a little hard to decipher, so, as an example, we write out the products when $h=3$.  The expression from Formula \ref{intcoefficients} is 
$ \| [0,a_1,\ldots,a_{k-3})\|\|[0,a_{k+1},\ldots,a_r)\| $ times 
\begin{align*} 
 \binom {a_k + a_{k-1} + a_{k-2}-1}{a_{k-1}+a_{k-2}}\binom {a_{k-1} + a_{k-2}-1}{a_{k-2}} 
\binom {a_{r:k-2}-1}{a_k + a_{k-1}+a_{k-2}},
\end{align*}  
which simplifies to 
\[
\frac {(a_{r:k-2}-1)!}{(a_k-1)!(a_{k-1}-1)!a_{k-2}! (a_k+a_{k-1} + a_{k-2})(a_{k-1}+a_{k-2}) (a_{r:k+1}-1)! }
\]
The expression from Corollary \ref{relativesymbols} is 
$ \| [0,a_1,\ldots,a_{k-3})\|\|[0,a_{k+1},\ldots,a_r)\| $ times 
\begin{align*} 
 \binom {a_{r:k+1} +a_k -1}{a_{k}-1}\binom {a_{r:k+1} + a_k + a_{k-1} -1}{a_{k-1}-1} 
\binom {a_{r:k+1} + a_k + a_{k-1} + a_{k-2} -1}{a_{k-2} -1}  
\end{align*} 
which simplifies to 
\[
\frac {(a_{r:k-2}-1)!}{(a_k-1)!(a_{k-1}-1)!(a_{k-2}-1)! (a_{r:k+1} +a_{k} + a_{k-1})(a_{r:k+1}+a_{k}) a_{r:k+1}! } \, .
\]
One could now easily see that the quotient of the two expression is indeed $m_{k,k-3} (\underline{a})$. 

We can use tableaux to visualize Formula \ref{gummybear}.  For example, the positions of $f_{r-4,r-4}$, $f_{r-3,r-4}$, $f_{r-2,r-4}$, $f_{r-1,r-4}$ are shown in the following tableau, with  the degree of $x$ and $y$ shown in the margin:
\[
\begin{array}{ccccccccccccc}
& \xi_{r-4} \! - \!  r \!+\! 1 &\xi_{r-4} \! - \! r \! + \! 2 & \xi_{r-4} \! - \! r \! + \! 3 & \xi_{r-4} \!-\! r \!+\! 4 
\\
\,\\
1\,&-\frac {a_r(a_{r:r-1})(a_{r:r-2})  }{(a_{r-1:r-3})(a_{r-2:r-3})a_{r-3}}  f_{r-4} 
\\
2&\,&  \frac {(a_{r:r-1}) (a_{r:r-2}) }{(a_{r-2:r-3})a_{r-3}}  f_{r-4}
\\ 
3\,&&&-\frac {(a_{r:r-2})}{a_{r-3}} f_{r-4} &&& 
\\
4 \,&&&&f_{r-4} 
\end{array}
\]
The leading and interior coefficients contributed by the ``flats'' at rank $r-h$ go northwest on a diagonal.  This point of view gives a variant of Formula \ref{gummybear}. 

\begin{formula}\label{gummybear1} 
\begin{align*}
& \bar\gamma(0,a_1,a_2, \ldots, a_r;x,y) = f_r(\underline{a}) T(U_{r,n}) 
\\
&  + 
(xy-x-y)\sum_{k=1}^{r-1}     \sum_{t=0}^{r-k-1}  (-1)^t f_{k+t,k}(\underline{a})(x-1)^{r-k-1} \tau(k+t+1, \xi_{k} -k - t - 1;y) .
\end{align*}
\end{formula}

We end with special cases of Formula \ref{gummybear} when the thickness is $1$ or $2$. 

\begin{cor} 
\begin{align*}
&\bar\gamma(1^{r-2},a_{r-1},a_r) = \frac {a_{r-1} \,T(U_{r,n}) }{\prod_{i=0}^{r-2} (a_{r-1}+a_r+i)} 
+  (xy-x-y)\frac {\tau(r,a_{r-1}-2;y)}{\prod_{i=1}^{r-2} (a_{r-1} +i)}
\\
\,\,
\\
&\bar\gamma(1^{r-3},a_{r-2},a_{r-1},a_r) = \frac {a_{r-2}a_{r-1}\,T(U_{r,n}) }{(a_{r-1}+a_r) \prod_{i=0}^{r-3}   (a_{r-2}+a_{r-1}+a_r + i) }
\\
&+ (xy-x-y)\left( \frac {a_{r-2}\,\tau(r,a_{r-2}+a_{r-1}-3;y)}{ \prod_{i=0}^{r-3}  (a_{r-2}+ a_{r-1} + i)}  + \frac{a_r\,\tau(r,a_{r-2}-3;y)}{(a_{r-1}+a_r) \prod_{i=1}^{r-3} (a_{r-2} +i)}\right) 
\\
&+  (xy-x-y)\frac {a_{r-1} \,\tau(r-1,a_{r-2}-2;y)}{(a_{r-1}+a_r) \prod_{i=1}^{r-3} (a_{r-2} +i)}
\end{align*}

\end{cor} 

\section{The M\"obius multiplier} \label{Meowbius} 

Recall that if $M$ is a rank-$r$ matroid on the set $S$, then the evaluation $T(M;1,0)$ equals the M\"obius invariant  $\mu(M)$.  When $M$ has no loops, $\mu(M)$ equals $(-1)^r \mu_M(\emptyset,S)$, where $\mu_M$ is the M\"obius function of the lattice of flats of $M$; when $M$ has loops, it equals $0$.  

\begin{thm}\label{Mobius}  
\[
\bar\gamma(0,a_1,\ldots,a_r; 1,0) = \frac {1}{\nu(0, a_r, a_{r-1}, \ldots, a_1)}.
\]
\end{thm}

\begin{cor} \label{Meow1} 
If $D$ is a PMD with flag composition $0,a_1,\ldots,a_r$, then 
\[
T(D;1,0) = \frac {\nu(0,a_1,a_2, \ldots,a_r)} {\nu(0,a_r, a_{r-1}, \ldots,a_1)}.  
\]
\end{cor}  
\noindent 
For example, the uniform matroid $U_{r,n}$ has flag composition $0,1,\ldots,1,n-r+1$ and the corollary asserts, correctly, that 
\[
T(U_{r,n};1,0)= \frac { n (n-1)(n-2)  \cdots (n-r+1)}{ n (r-1)(r-2)\cdots 1} = \binom {n-1}{r-1}.
\]

We begin the proof of Theorem \ref{Mobius} with the evaluation of a specialized symbol at $x=1, y=0$. 
\begin{lemma} \label{Meow2} 
$\displaystyle \,\,
[10^{a_1-1} \ldots 10^{a_r-1};1,0]  = \frac {(-1)^{a_r-1}}{n!}  \binom {n-1}{a_r-1}
$
\end{lemma}  
\begin{proof}  When we specialize the symbol $[10^{a_1-1}\ldots 10^{a_r-1}]$, there is always a factor of $(x-1)$ until we go pass the right-most $1$-bit $1_r$  at position 
$\xi_{r-1}+1$.  Thus,  
\begin{align*} 
n! [10^{a_1-1} \ldots 10^{a_r-1};1,0]  & = \sum_{m=a_r+1}^{n}  (-1)^{m} \binom {n}{m} 
\\
&= \sum_{m=0}^{a_r-1}  (-1)^{n+m} \binom {n}{m} = (-1)^{n+a_r-1} \binom {n-1}{a_r-1},
\end{align*} 
where we used Identity \ref{binomialID} in the last step.  
\end{proof} 
\noindent
Note that $[10^{a_1-1} \ldots 10^{a_r-1};1,0]$ depends only on $a_r$.   
By definition, 
\[
\bar\gamma(0,a_1,\ldots,a_r;1,0) = \underline{a}!  \sum_{\underline{b} \in [\underline{a})}  \mathsf{c}_{\underline{a}}(0,s_2(\underline{b},\underline{a}),\ldots,s_r(\underline{b},\underline{a}) )[\underline{b};1,0].
 \]
We will break up this sum into smaller sums over slices.  Let $S(\sigma_2,\ldots,\sigma_m)$  be the slice $[\underline{a};s_2=\sigma_2, \ldots,s_{m}=\sigma_{m})$ in which coordinates $2$ to $m$ are fixed.  The first coordinate $s_1$ is always $0$  as we are assuming $a_0=0$. Then 
\[
S(\sigma_2,\ldots,\sigma_m) = \bigcup_{i=0}^{a_m+\sigma_m-1} S(\sigma_2,\ldots,\sigma_{m-1},i).
\]
Using this, we have 
\begin{align*} 
&\sum_{\underline{b} \in S(\sigma_2,\ldots,\sigma_{r-1})}  \mathsf{c}_{\underline{a}}(\underline{b}) [\underline{b};1,0] 
=
\sum_{i=0}^{a_{r-1} +\sigma_{k-1} - 1}  \mathsf{c}_{\underline{a}}(0,\sigma_2,\ldots,\sigma_{r-1},i) [\underline{b};1,0] 
\\
&\qquad = \sum_{i=0}^{a_{r-1}+\sigma_{k-1}-1}  \frac {(-1)^{n+a_r+i} \mathsf{c}_{\underline{a}}(0,\sigma_2, \ldots,\sigma_{r-1},0)}{n!} \binom {a_r+i-1}{i} \binom {n-1}{a_r+i-1} 
\\
&\qquad  = \binom {a_{1:r} -1}{a_r-1}\frac{ \mathsf{c}_{\underline{a}}(0,\sigma_2,\ldots,\sigma_{r-1},0)}{n!} \sum_{i=0}^{a_{r-1}+\sigma_{r-1}-1}  (-1)^{n+a_r+i} \binom {a_{1:r-1}}{i} 
\\
&\qquad =\binom {a_{1:r}-1}{a_r-1} \frac { \mathsf{c}_{\underline{a}}(0,\sigma_2,\ldots,\sigma_{r-1},0)}{n!}  (-1)^{n+a_{r:r-1}+ \sigma_{r-1} -1}  \binom {a_{1:r-1}-1}{a_{r-1}+\sigma_{k-1}-1}
\end{align*} 
Going one step further, we have 
\begin{align*} 
&\sum_{\underline{b} \in S(\sigma_2,\ldots,\sigma_{r-2})}  \mathsf{c}_{\underline{a}}(\underline{b}) [\underline{b};1,0] =\sum_{i=0}^{a_{r-2}+\sigma_{r-2}-1}\sum_{\underline{b} \in S(\sigma_2,\ldots,\sigma_{r-2},i)}  \mathsf{c}_{\underline{a}}(\underline{b}) [\underline{b};1,0]
\\
& = \binom {a_{1:r}-1}{a_r-1} \frac {\mathsf{c}_{\underline{a}}(0,\sigma_2,\ldots,\sigma_{r-2},0,0)}{n!}  \sum_{i=0}^{a_{r-2}+\sigma_{r-2}-1}(-1)^{n+a_{r:r-1}+ i -1}  \binom {a_{r-1}+i-1}{i} \binom {a_{1:r-1}-1}{a_{r-1}+i-1}
\\
& = \binom {a_{1:r} -1}{a_r-1}\binom {a_{1:r-1}-1}{a_{r-1}-1} \frac {\mathsf{c}_{\underline{a}}(0,\sigma_2,\ldots,\sigma_{r-2},0,0)}{n!}  \sum_{i=0}^{a_{r-2}+\sigma_{k-2}-1}  (-1)^{n+a_{r:r-1}+i} 
\binom {a_{1:r-2}}{i} 
\\
& =\binom {a_{1:r}-1}{a_r-1} \binom {a_{1:r-1}-1}{a_{r-1}-1} \frac {\mathsf{c}_{\underline{a}}(0,\sigma_2,\ldots,\sigma_{r-2},0,0)}{n!}  (-1)^{n+a_{r:r-2}+ \sigma_{k-2} -1}  \binom {a_{1:r-2}-1}{a_{r-2}+\sigma_{k-2}-1}.
\end{align*} 
 
Continuing until we can go no further and noting that the signs cancel, we conclude that 
\[
\bar\gamma(0,a_1,\ldots,a_r;1,0) = \frac {\underline{a}!}{n!}  \prod_{i=2}^r \binom {a_{1:i} -1}{a_i-1}.
\]
As in Section \ref{slicenorm}, the binomial-coefficient product telescopes and after simplification, we obtain the formula in Theorem \ref{Mobius}.  

\begin{prop} \label{totalMobius} 
Let $M$ be an $(n,r)$-matroid with no loops.  Then 
\[
\sum_{0,a_1,a_2,\ldots,a_r}  \frac {\nu(M;  0,a_1,a_2,\ldots,a_r)}{\nu(0,a_r,a_{r-1},\ldots,a_1)}= \mu(M).
\]
\end{prop} 
\noindent 
The equation in Proposition \ref{totalMobius} is one leaf of a diptych, the other one being 
\[
\sum_{0,a_1,a_2,\ldots,a_r}  \frac {\nu(M;0,a_1,a_2,\ldots,a_r)}{\nu(0,a_1,a_2,\ldots,a_r)}= 1.
\] 
Both equations are easy consequences of Theorem \ref{positivity}.

Multiplying numerator and denominator by $a_{r:k-h+1}$, we obtain 
\[
m_{k,k-h} (\underline{a}) = \frac {\nu(0, a_{k-h+1},a_{k-h+2},\ldots, a_{k-1}, a_k, a_{r:k+1})}{\nu(0, a_{r:k+1},a_k,a_{k-1}, \ldots, a_{k-h+2} ,a_{k-h+1})}.
\]
The composition $a_{k-h+1},\ldots,a_k, a_{r:k+1}$ can be interpreted as the composition of the truncation 
a flag with composition $a_{k-h+1}, \ldots, a_k, a_{k+1}, \ldots, a_r$  in which the top $k$ flats coalesce.  Indeed, if the latter composition is the flag composition of a PMD, then the former composition is the flag composition of a truncation of that PMD.  
Thus, one may think of the M\"obius factor as the M\"obius invariant of a truncation.    

An easy calculation yields the following proposition.  

\begin{prop} \label{tripartition}  $f_{k+t,k}(\underline{a})  = m_{k+t,k} (\underline{a})f_{k}(\underline{a}) = $ 
\[
\frac {1}{\nu(0,a_1,\ldots,a_k) \nu(0,  a_{r:k+t+1},a_{k+t},a_{k+t-1}, \ldots,a_{k+2},a_{k+1})\nu(0,a_{k+t+1},a_{k+t+2}, \ldots,a_r)}.
\]
\end{prop} 
\noindent 
For example, when $r=12$, 
\[
f_{8,4} (\underline{a}) = \frac {1}{\nu(0,a_1,a_2,a_3,a_4)  \nu(0,a_{12:9},a_8,a_7,a_6,a_5)\nu(0,a_9,a_{10},a_{11},a_{12})}.
\]

\section{The flat expansion for the Tutte polynomial}  \label{FlatTutte} 

Let $M$ be an $(n,r)$-matroid with no loops. 
The \emph{flat numbers} $f_{k,m}^{t}$ of $M$ are defined for indices $k, \,m,\, t$ such that $k \leq m$ and $r-k\geq t \geq 1$ by 
\[
f_{k,m}^{t} = - \sum_{X: \rk(X) = k, \, |X| = m} \mu(\mathrm{Trun}^{r-k-t-1} (M/X)),
\]
where the sum ranges over all flats having rank $k$ and size $m$ and $\mu(\mathrm{Trun}^{r-k-t-1} (M/X))$ is the M\"obius invariant of the matroid obtained by repeatedly truncating 
the contraction $M/X$ until 
the result has rank $t+1$.   Note that $f^0_{k,m}$ is the number of flats $X$ such that $\rk(X)=k$ and $|X| = m$.  
The flat numbers form a $3$-dimensional array which we could call the \emph{flat tensor} or \emph{$f$-tensor} of the matroid.     
When we show a specific $f$-tensor, we shall only show the numbers that are non-zero.  In addition, for calculating the Tutte polynomial, we do not need the numbers $f_{k,k}^t$ contributed by the flats which are independent sets and we shall not show these numbers.  
\footnote{The numbers $f^0_{k,k}$ can be derived from the numbers $f_{k,m}^0, \, k < m$, but this seems not to be true when $t \geq 1$. } 
As an example, the $f$-tensor of $8U_{1,2}$, the direct sum of $8$ copies of two parallel elements,
is shown in the following tableau:
\[ 
\begin{array}{ccccccccccccc}    
\,  & &&& & &  && 1
\\
&-70_\ast && 56 & & -28& & 8
\\
&&210_\ast &   & -112& &  28
\\
&336&& -210_\ast & & 56
\\
&&-224&   &  70_\ast
\\
 &-140 && 56 
\\
&&28
\\ 
&8
\\
1
\end{array}
\]
where the blank entries are equal to $0$.  The flat numbers $f^0_{k,2k}$ are on the $45^{\circ}$ main diagonal and the numbers $f^t_{k,2k}$ go northwest on a $135^{\circ}$ diagonal instead of up in the third dimension.  The numbers $f_{4,8}^t$ are tagged with the marker $\ast$ as an example.

As one might expect, flat numbers are difficult to calculated in general and might require detailed knowledge of the lattice of flats of the matroid.  
However, flat numbers can be derived from the $\mathcal{G}(M)$ and they are valuative invariants.  For the numbers $f^0_{k,m}$, this is done in Section 5 of \cite{Catdata}.  For the other flat numbers, the arguments in that section, combined with Theorem \ref{Mobius}, will give a proof.  

The next formula is obtained by combining Theorem \ref{positivity} and Formula \ref{gummybear} and changing the order of summation.  

\begin{formula}\label{fexpansion} 
\begin{align*}
& T(M;x,y)  = T(U_{r,n};x,y) 
\\
&  + 
(xy-x-y)\sum_{k=1}^{r-1}  (x-1)^{r-k-1}  \left( \sum_{k,m}  \sum_{t=0}^{r-k-1} f_{k,m}^{t}  \tau(k+1+t, m - k - t - 1;y)\right).
\end{align*}
\end{formula}
\noindent 
Note that when $k=m$, $m-k-t-1 < 0$ and $\tau(k+1+t,m-k-t-1;y)$ is identically zero for all $t \geq 0$.  Thus, Formula \ref{fexpansion} uses only the flat numbers $f_{k,m}^t$ where $k < m$.  This is the reason for omitting the
numbers $f^t_{k,k}$ when we write down an $f$-tensor. 

Before giving the proof, we derive two special cases of Formula \ref{fexpansion}, giving formulas for the M\"obius invariant and the number of bases.  Both parameters involve evaluations of $T(M;x,y)$ at $x = 1$, and another value for $y$.  When $x=1$, $x-1 = 0$ and only the terms at rank $r-1$ in the summation  in Formula \ref{fexpansion} survive.  

\begin{cor}   Let $M$ be an $(n,r)$-matroid with no loops.  Then 
\begin{align*} 
&\mu(M) = T(M;1,0) = \binom {n-1}{r-1} + \sum_{X} \mu(M/X)  \binom {|X| -  1}{\rk(X)},
\\
& T(M,1,1) = \binom {n}{r} + \sum_{X} \mu(M/X)  \binom {|X|}{r},
\end{align*}
where the sums range over all flats having rank less than $r$.  
\end{cor} 
\noindent 
The second formula can be derived independently by M\"obius inversion, but the first seems to be a new recursion for the M\"obius invariant.

To prove Formula \ref{fexpansion}, we need to do a formal calculation with $\mathcal{G}$-invariants.  We first derive a result, a special case of which, unfortunately, was very carelessly stated
in \cite{Catdata} as Proposition 4.2.

\begin{prop} Let $M$ be an $(n,r)$-matroid and $a_0,a_1, \ldots a_r$ an $(n,r)$-composition. 
Then 
\[
\nu(\mathrm{Trun}^{r-t}(M); a_0,a_1,\ldots,a_{t-1}, a_{t:r}) = \sum_{a_t,a_{t+1}, \ldots,a_r} \frac {\nu(M;a_0,a_1,\ldots,a_r)}{\nu(0,a_{t},a_{t+1},\ldots,a_r)}
\]
and hence, 
\[
\mathcal{G}(\mathrm{Trun}^{r-t}(M)) = \sum_{\underline{a}} \nu(M;a_0,a_1,\ldots,a_r) \frac  {\gamma(a_0,a_1,\ldots,a_{t-1},a_{t:r})}{\nu(0,a_{t},a_{t+1},\ldots,a_r)} .
\]
\end{prop}
\begin{proof}  Let $a_0,a_1,\ldots,a_{t-1}$ be an $(m,t-1)$-composition (so that $a_{0:t-1}=m$). Then  
\begin{align*}
& \sum_{b_{t},\ldots,b_r} \nu(M;a_0,a_1,\ldots,a_{t-1},b_{t},\ldots,b_r) \gamma(0,b_{t},\ldots,b_r) 
\\
& \qquad = \sum_{X: \rk(X) = t-1, |X| = m} \nu(M|X; a_0,a_1,\ldots, a_{t-1},a_t)  \mathcal{G}(M/X),
\end{align*} 
where $0,b_{t},\ldots,b_r$ ranges over all $(n-m,r-t+1)$-compositions.  This equation follows from the fact that each flag with composition $a_0,a_1, \ldots, a_{t-1}, b_{t},\ldots, b_r$ goes through a unique flat $X$ of rank $t-1$ and size $m$.
We will now apply the specialization 
\[
\gamma(0,b_{t},\ldots,b_r) \mapsto \frac {(n-m)!}{\nu(0,b_{t},\ldots,b_r)}
\]
specified on the $\gamma$-basis.  This is the specialization $[\underline{b}] \mapsto 1$ in the symbol basis and hence, $\mathcal{G}(M/X)$ specializes to $(n-m)!$.  
The specialized equation says 
\begin{align*}
\sum_{0,b_{t},\ldots,b_r} \frac {\nu(M;a_0,a_1,\ldots,a_{t-1},b_{t},\ldots,b_r)}{ \nu(0,b_{t},\ldots,b_r)} &= \sum_{X} \nu(M|X; a_0,a_1,\ldots,a_t)
\\
&= \nu(\mathrm{Trun}^{r-t}(M);a_0,a_1,\ldots,a_{t-1},a_{t:r}),
\end{align*} 
completing the proof.
\end{proof} 

We note as a corollary the case when we truncate once.  This gives the correct statement of Proposition 4.2 in \cite{Catdata}. 

\begin{cor}
\[
\mathcal{G}(\mathrm{Trun}(M)) = \sum_{\underline{a}} \nu(M;a_0,a_1, \ldots, a_r)  \frac { a_{r-1} \gamma(a_0,a_1, \ldots, a_{r-2},a_{r-1}+a_r)}{a_{r-1}+ a_r} . 
\]
\end{cor} 

Let $M$ be an $(n,r)$-matroid and $k$, $t$, and $m$ integers such that $m \geq k > 0$, $t \geq 0$, $n \geq k+t$.   Then 
\begin{align*}
&\mathcal{G}(M|X) = \sum_{a_0,a_1,\ldots,a_k}  \nu(M|X;a_0,a_1,\ldots,a_k) \gamma(a_0,a_1,\ldots,a_k), 
\\
& \mathcal{G}(\mathrm{Trun}^{r-k-t-1} (M/X)) 
\\
&\qquad = \sum_{0,a_{k+1},\ldots,a_r} \frac {\nu(M/X;0,a_{k+1},\ldots,a_r)}{\nu(0,a_{k+t+1},a_{k+t+2},\ldots,a_r)}  
\gamma(0,a_{k+1},\ldots,a_{k+t}, a_{k+t+1:r}).
\end{align*}
Multiplying formally (treating the $\gamma$-basis elements as indeterminates) and summing over all flats $X$ such that $ \rk(X)=k, |X|=m$,
we obtain the following equation ($*$): 
\begin{align*}
&\sum_{X}  \mathcal{G}(M|X) \mathcal{G}(\mathrm{Trun}^{r-k-t-1} (M/X)) 
\\
& = \sum_{X} \frac {\nu(M|X; a_0,\ldots,a_k) \nu(M/X;0,a_{k+1},\ldots,a_r)}{\nu(0,a_{k+t+1},a_{k+t+2},\ldots,a_r)}  \gamma(a_0,a_1,\ldots,a_k)\gamma(0,a_{k+1}\ldots,a_{k+t},a_{k+t+1:r})
\\
&= \sum_{\underline{a}:a_{0:k}=m} \frac {\nu(M; a_0,\ldots,a_k,a_{k+1},\ldots,a_r)}{\nu(0,a_{k+t+1},a_{k+t+2},\ldots,a_r)}  \gamma(a_0,a_1,\ldots,a_k)\gamma(0,a_{k+1}.\ldots,a_{k+t},a_{k+t+1:r}).
\end{align*}
In the final step, we use the fact that a flag going through the flat $X$ is a unique concatenation of a flag going up to $X$ and a flag going up from $X$ to obtain for 
a given $(m,k)$-composition $a_0,a_1,\ldots,a_k$,  
\[
\nu(M; a_0,\ldots,a_k,a_{k+1},\ldots,a_r) = \sum_{X} \nu(M|X;a_0,a_1,\ldots,a_k)   \nu(M/X;0,a_{k+1}, \ldots, a_r) .
\] 

We will now assume $a_0=0$ and consider the specializations 
\begin{align*} 
& \gamma(0,a_1,\ldots,a_k) \mapsto \frac {a_{1:k}! }{\nu(0,a_1,\ldots,a_k)} 
\\
& \gamma(0,a_{k+1},\ldots,a_{k+t},a_{k+t+1:r}) \mapsto 
\frac {1}{\nu(0,a_{k+t+1:r},a_{k+t},\ldots,a_{k+1}) } .
\end{align*} 
These specializations, specified for the $\gamma$-basis, equal the specializations 
\[
[\underline{b}] \mapsto 1, \qquad [\underline{b}] \mapsto [\underline{b};0,1] 
\]
specified for the symbol basis.  Hence, under these specializations, $\mathcal{G}(M|X) \mapsto a_{1:k}!$.   Performing these specializations on equation ($*$) yields the following corollary, which is exactly what is needed to verify Formula \ref{fexpansion}. 

\begin{cor}  Let $M$ be an $(n,r)$-matroid with no loops.  Then 
\[
\sum_{\underline{a}:a_0=0,\,a_{1:k} =m} \nu(M;\underline{a}) m_{k+t, k} (\underline{a})f_{k,k}(\underline{a}) = f_{k,m}^t , 
\]
where the sum ranges over all $(n,r)$-compositions such that $a_0=0$ and $a_1+\cdots+a_k = m$.    
\end{cor}

The $f$-tensor contains more information than is necessary to determine the Tutte polynomial.  
Let the \emph{total flat numbers} $F_{ij}$ be defined by 
\[
F_{ij} = \sum_{m-k-t=j,\, r-k+t=i}  f_{km}^t
\]
and the \emph{$F$-tableau} be the tableau with the $ij$-entry equal to $F_{ij}$.   The next formula is a ``trivial'' rewriting of Formula \ref{fexpansion}; it's the corollary that is interesting. 

\begin{formula} \label{Fexpansion}  Let $M$ be an $(n,r)$-matroid with no loops.  Then 
\[
T(M;x,y) = T(U_{r,n};x,y) + (xy-x-y)\sum_{i,j}  F_{ij} (x-1)^{r-i} \tau(i+1,j-1 ;y). 
\]
\end{formula} 

Because the polynomials $(xy-x-y)(x-1)^{r-i}  \tau(i+1,j-1;y)$ are linearly dependent in the vector space $\mathbb{K}[x,y]$, the total flat numbers $F_{ij}$ are determined by 
$T(M)$.  Thus, we have proved following corollary for matroids with no  loops. The case when there are loops is easily dealt with.    

\begin{cor} \label{codetermination} Let $M$ be an $(n,r)$-matroid.  then the Tutte polynomial $T(M)$ determines and is determined by its $F$-tableau.  
\end{cor} 

Formula \ref{fexpansion} involves M\"obius invariants of upper intervals, and so does the coboundary or bad-koala (that is, bad-colouring) polynomial form of the corank-nullity polynomial.  I don't think one can be easily be derived from the other.  

\section{A cellular automaton for computing Tutte polynomials?} 

A distinctive feature of the formulas in this paper is the presence of the polynomial $xy - x -y$.  This is not accidental.  In \cite{syzygy},  it is shown that all linear relations among the coefficients of Tutte polynomials of $(n,r)$-matroids follow from the fact that the difference $T(M_1)-T(M_2)$, where $M_1$ and $M_2$ are $(n,r)$-matroids, is divisible by $xy-x-y$.  
In tableau form, $x^i y^j (xy - x - y)$ is  
\[
\begin{array}{ccc}
&-1
\\
-1&1
\end{array}
\]
where $1$ is the coefficient of $x^{i+1} y^{j+1} $.  
This rule describes a cellular automaton in the plane, in which adding $1$ in the coordinate $(i+1,j+1)$, where $i$ and $j$ are positive integers, requires subtracting $1$ from the coordinates $(i,j+1)$ and $(i+1,j)$.   Examples should make it clear how this works. 

Suppose we wish to calculate $T(\mathrm{PG}(3,2))$, the Tutte polynomial of the ternary projective plane.  Since $f^0_{2,4} = 13$, Formula \ref{fexpansion} says that it equals 
\[
T(U_{3,13}) + 13(xy-x-y)(1+y).   
\]
In tableau form, $T(U_{3,13})$ is 
\[
\begin{array}{ccccccccccc}   
\,  &  56 & 45  & 36 & 28 & 21  & 15& 10 & 6 & 3  & 1
\\
56 &  
\\
10
\\
1
\end{array}\]
and $13(xy-x-y)(1+y)$ is 
\[
\begin{array}{ccccccccccc}   
\,  &  -39 & -13 &\qquad\qquad\qquad\qquad\qquad\qquad &
\\
-39 &  26 &   13\,.
\end{array}\]
Thus, $T(\mathrm{PG}(2,3))$ equals 
\[\begin{array}{ccccccccccc}   
\,  &  16 & 32   & 36 & 28 & 21  & 15& 10 & 6 & 3  & 1
\\
16 &  26 &   13
\\
10
\\
1
\end{array}.
\]
As a check, the M\"obius invariant (which equals the sum of the coefficients on the zeroth column) is indeed $27$.  

To do the next case,  the ternary projective space $\mathrm{PG}(3,3)$, we note that  $x^iy^j (x-1)(xy-x-y)$ translated into tableau form, is
\[
\begin{array}{cccc}
&&1
\\
&1& -2
\\
&-1&1
\\
&
\end{array} 
\]
and, although we do not need it, the translation of $x^iy^j (x-1)^2(xy-x-y) $ is 
\[
\begin{array}{cccc}
&-1
\\
-1&3
\\
2& -3
\\
-1& 1&.
\end{array}
\]
There are similar translations for higher powers of $(x-1)$.  

We can now take on $\mathrm{PG}(3,3)$. 
For this matroid, $n=40$, $r=4$, $n-r=36$, $f_{3,13} = 40$, $f_{2,4} = 130$, $ f_{2,4}^{1} = 390$. 
We begin the calculation with $T(U_{4,40})$.  It equals 
\[ 
\begin{array}{ccccccccccccccccc}   
\,  &   8436 & 7770 & 7140& 6545& 5984& 5456 & 4960 &4495  & 4060 & 3654 & \mathsf{R}
\\
8436 
\\
666 
\\
36  
\\
1
\end{array}
\]
where $\mathsf{R}$ is the row with $26$ binomial coefficients
\[
3276\,\,2925\,\,2600\,\,2300\,\,2024\,\qquad\,\,\ldots\qquad       \,\,  35\,\,\,\, 20 \,\,\,\,10\,\,\,\,\,4\,\,\,\,\, 1
\]
or
\[
\quad \tbinom {28}{25}  \,\,\,\,\tbinom {27}{24} \,\,\,\,\tbinom {26}{23} \,\,\,\, \tbinom{25}{22} \,\,\,\, \tbinom{24}{21}\,\qquad\ldots\qquad       
\,   \tbinom {7}{4}  \,\,\tbinom {6}{3} \,\,\tbinom {5}{2} \,\, \tbinom{4}{1} \,\, \tbinom{3}{0}\,.
\]
Next we write down $40(xy-x-y)\tau(4,9;y)$:
\[
\begin{array}{ccccccccccccccccc}   
&   -8800     & -6600 & -4800 & -3360 & -2240& -1400 & -800& -400 & -160 & -40
\\
-8800 &   2200      & 1800 & 1440 & 1120& 840 & 600 & 400 & 240 & 120 & 40\,
\end{array}
\]
then, $-390(xy-x-y)\tau(4,0;y)$:
\[
\begin{array}{ccccccccccccccccc}   
&   390     &&& \hskip 3.5in
\\
390 &   -390        &
\end{array}
\]
and finally $130(xy-x-y)(x-1)\tau(3,1;y)$:
\[
\begin{array}{ccccccccccccccccc}   
\,  &   390      & 130 &&&\hskip3.0in
\\
390 &   -650       & -260
\\
-390 &260  &130 \, .
\end{array}
\]
Adding up the tableaux, we have the answer: $T(\mathrm{PG}(3,3))$ equals 
\[ 
\begin{array}{ccccccccccccccccc}   
\,  &   416& 1230 & 2340 & 3185 & 3744& 4056 & 4160 & 4095 & 3900 & 3614 &\mathsf{R} 
\\
416 &  1160 &   1540 & 1440 &1120 &840&  600 &400 &240 &120& 40
\\
276& 260 &130
\\
36  
\\
1 &&
\end{array}.
\]
As a check,  the M\"obius invariant works out to be $729$, as expected.  If we display the $f$-tensor with the numbers $f^t_{k,m}$ going northwest,  then all we have is 
\[
\begin{array} {ccccccccccccc}  
\, &&&& \hskip 1.0in &  \hskip .7in \ldots\ldots & \hskip 1in & 1 
\\
\,  & 390 &&\ldots &40
\\\,&  &130 
\end{array} .
\] 
This is an example of the fact that for many matroids, the Tutte polynomial contains information only at a few locations, with the rest being binomial-coefficient ``filler''.  This is rather like the genetic code, according to current biology. 

We end with the $f$-tensor for $M(K_7)$, displayed as a tableau (with the top row giving the degree of $y$ for readability):
 {\small
\[
\begin{array}{cccccccccccccccccc} 
&1&2&3&4&5& 6& 7& 8& 9& 10&  11& 12& 13&  14& 15 
\\
\, \\
&.&.&.&.&.& .& .& . & .& .&  .& . & .&  .& 1 
\\
.& 210_\star \, | -140_\circ  &-210_\bullet&.& 35&  -42_\ast  & 21& .& .& .&7
\\
. & .&  -175_\star  \, |  \,70_\circ &105_\bullet&&& 21_\ast
\\
.& 240 & . & 35_\star 
\\
. & 35 
\\
.
\\
.
\end{array} .
\]    
}
Because we are showing a $3$-dimensional tensor with the third dimension going northwest, there are entries having two numbers; the numbers $f^t_{k,m}$ with the same double subscripts are tagged with the same marker, so that, for example, $-42_\ast$ is $f_{4,5}^1$ and $21_\ast$ is $f_{4,5}^0$. The $F$-tableau is obtained by adding up numbers at the same entry.  
It is now possible to calculate $T(M(K_7))$ with Formula \ref{Fexpansion} and finish as we commenced.

\section{Appendix: Examples of $\bar\gamma$-expansions and $f$-tensors}

\subsection{The matroid $H_{3,3}$} The matroid $H_{3,3}$ shown on the left in Figure \ref{fig:matroidH}.  It is a direct sum, and 
\[
T(H_{3,3}) = (x +y + y^2) \big(x^2 + x(1+y+y^2)+ (y+y^2+y^3)\big).
\]
Expanded as a tableau, 
\[
T(H_{3,3}) = 
\begin{array}{cccccc}   
\,  &  \,& 1 & 2 & 2 & 1
\\
\, &  2& 3 &3 & 1
\\
1& 2 &  2
\\
1
\end{array}.
\]
Its $f$-tensor is given by $f^0_{2,5}=1$, $f^0_{2,4}=2$, $f_{1,2}^0=2$, and $f_{1,2}^1=3$; its catenary data are given by 
\[
\nu(H_{3,3};3,3,2)=2, \nu(H_{3,3};3,2,3)=1, \nu(H_{3,3};3,1,4)=2, \nu(H_{3,3};1,3,4)=2, \nu(H_{3,3};1,4,3)=2.
\]

\showhide{
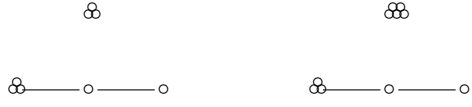
\begin{figure}
  \centering
  \begin{tikzpicture}[scale=1]  
\node[inner sep = 0.3mm] (d13) at (8.5, 0.95) {\small $\circ$};\node[inner sep = 0.3mm] (d14) at (8.6, 0.95) {\small $\circ$}; 
\node[inner sep = 0.3mm] (d15) at (8.55, 1.05) {\small $\circ$};
\node[inner sep = 0.3mm] (d23) at (9.5, .95) {\small $\circ$};
\node[inner sep = 0.3mm] (d33) at (10.5, .95) {\small $\circ$};
\foreach \from/\to in {d13/d23,d23/d33} \draw(\from)--(\to);
\node[inner sep = 0.3mm] (d24) at (9.5, 1.95) {\small $\circ$};\node[inner sep = 0.3mm] (d24a) at (9.6, 1.95) {\small $\circ$}; 
\node[inner sep = 0.3mm] (d24b) at (9.55, 2.05) {\small $\circ$};\node[inner sep = 0.3mm] (d24b) at (9.65, 2.05) {\small $\circ$};
\node[inner sep = 0.3mm] (d24b) at (9.7, 1.95) {\small $\circ$};
\node[inner sep = 0.3mm] (c13) at (4.5, 0.95) {\small $\circ$};\node[inner sep = 0.3mm] (c14) at (4.6, 0.95) {\small $\circ$}; 
\node[inner sep = 0.3mm] (c15) at (4.55, 1.05) {\small $\circ$};
\node[inner sep = 0.3mm] (c23) at (5.5, .95) {\small $\circ$};
\node[inner sep = 0.3mm] (c33) at (6.5, .95) {\small $\circ$};
\foreach \from/\to in {c13/c23,c23/c33} \draw(\from)--(\to);
\node[inner sep = 0.3mm] (c14) at (5.5, 1.95) {\small $\circ$};\node[inner sep = 0.3mm] (c14a) at (5.6, 1.95) {\small $\circ$}; 
\node[inner sep = 0.3mm] (c14b) at (5.55, 2.05) {\small $\circ$};
  \end{tikzpicture}
 \caption{The matroids $H_{3,3}$ and $H_{3,5}$} 
 \label{fig:matroidH}
\end{figure}}

The frame elements occurring in the $\bar\gamma$-expansion of $T(H_{3,3})$ are 

\begin{eqnarray*}
\bar\gamma(3,3,2) &=&
\tfrac {9}{40}T(U_{3,8}) + (xy-x-y)\left[\tfrac {1}{2}(10+6y+3y^2+y^3) + \tfrac {3}{5} (x-1)(2+y) - \tfrac {2}{5}\right] 
\\
\bar\gamma(3,2,3) &=&
\tfrac {6}{40}T(U_{3,8}) + (xy-x-y)\left[\tfrac {3}{5}(6+3y+y^2) + \tfrac {2}{5} (x-1)(2+y) - \tfrac {3}{5}\right] 
\\
\bar\gamma(3,1,4) &=&
\tfrac {3}{40}T(U_{3,8}) + (xy-x-y)\left[\tfrac {3}{4}(3+y) + \tfrac {1}{5} (x-1)(2+y) - \tfrac {4}{5}\right] 
\\
\bar\gamma(1,3,4) &=&
\tfrac {3}{56}T(U_{3,8}) + (xy-x-y)\left[\tfrac {1}{4}(3+y)\right] 
\\
\bar\gamma(1,4,3) &=&
\tfrac {4}{56}T(U_{3,8}) + (xy-x-y)\left[\tfrac {1}{5}(6+3y+y^2)\right].
\end{eqnarray*}

\subsection{The matroid $H_{3,5}$} 
The matroid $H_{3,5}$ shown on the right in Figure \ref{fig:matroidH}.  It is also a direct sum,
and 
\begin{align*} 
T(H_{3,5}) & = (x +y + y^2+ y^3 + y^4) \big(x^2 + x(1+y+y^2)+ (y+y^2+y^3)\big).
\\
&= 
\begin{array}{ccccccccc}
\,  &  \,& 1 & 2 & 3 & 3 & 2 & 1
\\
\, &  2& 3  &4  & 3 & 2 & 1
\\
1 & 2 &  2 &1 & 1
\\
1
\end{array}.
\end{align*} 
The $f$-tensor is given by $f^0_{2,8}=1$, $f^0_{2,6}=2$, $f^0_{2,5}=1$, $f^0_{1,5}=1$, $f^1_{1,5}=2$, $f^0_{1,3}=1$, and $f^1_{1,3}=1$. 
Its catenary data are given by 
\begin{eqnarray*}
\nu(H_{3,5};5,3,2)=1, \nu(H_{3,5};3,5,2)=1, \nu(H_{3,5};5,1,4)=2, 
\\
\nu(H_{3,5};3,2,5)=1, \nu(H_{3,5};1,5,4)=2, \nu(H_{3,5};1,4,5)=2.
\end{eqnarray*}
The frame elements occurring in the $\bar\gamma$-expansion of $T(H_{3,5})$ are 

\begin{eqnarray*}
\bar\gamma(5,3,2) &=&
\tfrac {3}{10}T(U_{3,10}) + (xy-x-y)\left[\tfrac {5}{8}(21 + 15y +10y^2 +6y^3 +3y^4+y^5)  \right. 
\\
&& \left. \qquad\qquad  + \tfrac {3}{5} (x-1)(4+3y+2y^2+y^3) - \tfrac {2}{5}(6+3y +y^2)  \right] 
\\
\bar\gamma(3,5,2) &=&
\tfrac {15}{70}T(U_{3,10}) + (xy-x-y)\left[\tfrac {3}{8}(21 + 15y +10y^2 +6y^3 +3y^4+y^5)   \right.
\\
&&\left. \qquad\qquad + \tfrac {2}{7} (x-1)(2+y) - \tfrac {5}{7}\right] 
\\
\bar\gamma(3,2,5) &=&
\tfrac {6}{70}T(U_{3,10}) + (xy-x-y)\left[\tfrac {3}{5}(6+3y+y^2) + \tfrac {2}{7} (x-1)(2+y) - \tfrac {5}{7} \right] 
\\
\bar\gamma(5,1,4) &=&
\tfrac {1}{10}T(U_{3,10}) + (xy-x-y)\left[\tfrac {5}{6}(10+6y +3y^2+y^3)  \right.
\\
&& \left. \qquad\qquad  + \tfrac {1}{5} (x-1)(4+3y+2y^2+y^3)  - \tfrac {4}{5}(6+3y+y^2) \right] 
\\
\bar\gamma(1,5,4) &=&
\tfrac {5}{90}T(U_{3,10}) + (xy-x-y)\left[\tfrac {1}{6}(10+6y+3y^2 +y^3)\right] 
\\
\bar\gamma(1,4,5) &=&
\tfrac {4}{90}T(U_{3,10}) + (xy-x-y)\left[\tfrac {1}{5}(6+3y+y^2)\right].
\end{eqnarray*}

\subsection{Two row-echelon matroids}

\showhide{
\begin{figure}
  \centering
  \begin{tikzpicture}[scale=1]  
\node[inner sep = 0.3mm] (c13) at (6.5, 0.95) {\small $\circ$};\node[inner sep = 0.3mm] (c14) at (6.6, 0.95) {\small $\circ$}; 
\node[inner sep = 0.3mm] (c15) at (6.55, 1.05) {\small $\circ$};  \node[inner sep = 0.3mm] (c16) at (6.65, 1.05) {\small $\circ$};
\node[inner sep = 0.3mm] (c23) at (7.2, .85) {\small $\circ$};
\node[inner sep = 0.3mm] (c33) at (7.7, 1.25) {\small $\circ$};
%
%
\node[inner sep = 0.3mm] (c14) at (7.0, 1.95) {\small $\circ$};\node[inner sep = 0.3mm] (c14) at (8.0, 1.65) {\small $\circ$}; 
\node[inner sep = 0.3mm] (c14) at (7.55, 2.45) {\small $\circ$};   \node[inner sep = 0.3mm] (c14) at (7.9, 2.15) {\small $\circ$};
 %
 %
 \node[inner sep = 0.3mm] (d13) at (10.5, 0.95) {\small $\circ$};\node[inner sep = 0.3mm] (d14) at (10.6, 0.95) {\small $\circ$}; 
\node[inner sep = 0.3mm] (d15) at (10.55, 1.05) {\small $\circ$};  \node[inner sep = 0.3mm] (d16) at (10.65, 1.05) {\small $\circ$};
\node[inner sep = 0.3mm] (d23) at (11.5, .95) {\small $\circ$};
\node[inner sep = 0.3mm] (d33) at (12.5, .95) {\small $\circ$};
\foreach \from/\to in {d13/d23,d23/d33} \draw(\from)--(\to);
\node[inner sep = 0.3mm] (d14) at (11.0, 1.95) {\small $\circ$};\node[inner sep = 0.3mm] (d14) at (11.8, 1.55) {\small $\circ$}; 
\node[inner sep = 0.3mm] (d14) at (11.55, 2.45) {\small $\circ$};   \node[inner sep = 0.3mm] (d14) at (11.9, 2.15) {\small $\circ$};
 
 \end{tikzpicture}
 \caption{The row-echelon matroids $F(1000110000)$ and $F(1000101000)$} 
 \label{fig:matroidF}
\end{figure}
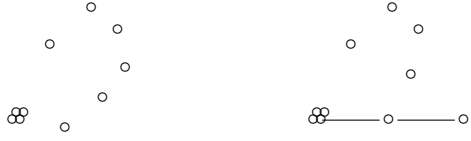}

Let $F_1$ be the row-echelon matroid 
\footnote{Row-echelon matroids are matroids which can be represented by generic row-echelon matrices.  They occur naturally, first in \cite{MR0190045}, and were called nested matroids, Schubert matroids, or freedom matroids.  I propose yet another name: row-echelon matroids.  They could also be called {\bf G}eneric {\bf row}-eche{\bf l}on or {\bf Growl} matroids; my cat would like that. }
$F(1000110000)$ shown on the left in Figure \ref{fig:matroidF}.    
Its $f$-tensor is given by $f_{2,5}^0 = 6$, $f^0_{1,4}=1$, and $f_{1,4}^1 =5$ and 
\[ 
T(F(1000110000)) = \begin{array}{cccccccc}   
\,  &  10& 10  & 10 & 10 & 6 & 3  & 1
\\
10 &  4&   4   & 4 
\\
4  &   1  &  1 & 1
\\
1
\end{array}.
\]
The catenary data are  given by 
\[
\nu(F_1;4,1,5)=6, \nu(F_1;1,4,5)=6, \nu(F_1;1,1,8)=30.
\]

The matroid $F_2$ is the row-echelon matroid $F(1000101000)$ shown on the right in Figure \ref{fig:matroidF}.  Its $f$-tensor is given by 
$f_{2,6}^0 = 1$, $f^0_{2,5} = 4$, $f^0_{1,4}=1$, and $f_{1,4}^1 =4$ and 
\[ T(  F(1000101000))= 
\begin{array}{ccccccccc}   
\,  &  9& 9  & 9 & 9 & 6 & 3  & 1
\\
9 &  4&   4   & 4 & 1
\\
4  &   1  &  1 & 1
\\
1
\end{array}.
\]
The catenary data are  given by 
\[
\nu(F_2;4,2,4)=1, \nu(F_2;4,1,5)=4, \nu(F_2;1,5,4)=2, \nu(F_2;1,4,5)=4, \nu(F_2;1,1,8)=28.
\]

The frame elements occurring in $\bar\gamma$-expansion of the two Tutte polynomials are

\begin{eqnarray*}
\bar\gamma(4,2,4) &=&
\tfrac {2}{15}T(U_{3,10}) + (xy-x-y)\left[\tfrac {2}{3}(10+6y+3y^2+y^3)  \right.
\\
&& \qquad\qquad\qquad \left.  + \tfrac {1}{3} (x-1)(3+2y+y^2) - \tfrac {2}{3}(3+y)\right] 
\\
\bar\gamma(4,1,5) &=&
\tfrac {1}{15}T(U_{3,10}) + (xy-x-y)\left[\tfrac {4}{5}(6+3y+y^2)     \right.
\\
&& \qquad\qquad\qquad \left. + \tfrac {1}{6} (x-1)(3+2y+y^2) - \tfrac {5}{6}(3+y) \right] 
\\
\bar\gamma(1,5,4) &=&
\tfrac {5}{90}T(U_{3,10}) + (xy-x-y)\left[\tfrac {1}{6}(10+ 6y+ 3y^2+ y^3)\right] 
\\
\bar\gamma(1,4,5) &=&
\tfrac {4}{90}T(U_{3,10}) + (xy-x-y)\left[\tfrac {1}{5}(6+3y+y^2)\right]
\\
\bar\gamma(1,1,8) &=&
\tfrac {1}{90}T(U_{3,10}). 
\end{eqnarray*}

 \subsection{Boolean algebras with multiple points} 

Boolean algebras with multiple points contains copies of $\gamma$-basis or $\bar\gamma$-frame elements within the simplest matroid structures/. They should be useful in some way. In the table are the coefficients for the $\bar\gamma$-frame elements $\bar\gamma(0,a_1,a_2,a_3;x,y)$ contained in 
$ U_{1,2} \oplus U_{1,3} \oplus U_{1,5} $.
\begin{center}  
\begin{tabular}{lclclclclclclcl}
\hline 
Composition   \rule{0pt}{15pt} &       $1/\nu$ & $\xi_2$  & ${f}_2$  & $\xi_1$ & ${f}_1$  & $f_{2,1}$
\\
\hline
\hline 
$2,3,5$  \rule{0pt}{15pt} &  $\tfrac {6}{80}$  &    $5$  & $\tfrac {2}{5}$  & $2$ & $\tfrac {3}{8}$ & $\tfrac{5}{8}$ 
\\
\hline 
$3,2,5$  \rule{0pt}{15pt} &  $\tfrac {6}{70}$  &    $5$  & $\tfrac {3}{5}$  & $3$ & $\tfrac {2}{7}$ & $\tfrac{5}{7}$ 
\\
\hline 
$2,5,3$  \rule{0pt}{15pt} &  $\tfrac {10}{80} $ &      $7$  & $\tfrac {2}{7}$  & $2$ & $\tfrac {5}{8}$ & $\tfrac{3}{8} $
\\
\hline 
$5,2,3$  \rule{0pt}{15pt} &  $\tfrac {10}{50}$  &    $7$  & $\tfrac {5}{7}$  & $5$ & $\tfrac {2}{5}$ & $\tfrac{3}{5} $
\\
\hline 
$3,5,2$  \rule{0pt}{15pt} &  $\tfrac {15}{70}$  &    $8$  & $\tfrac {3}{8}$  & $3$ & $\tfrac {5}{7}$ & $ \tfrac{5}{8} $
\\
\hline 
$5,3,2$  \rule{0pt}{15pt} &  $\tfrac {15}{50}$  &    $8$  & $\tfrac {5}{8}$  & $5$ & $\tfrac {3}{5}$ & $\tfrac{2}{5} $
\\
\hline
\end{tabular}
\end{center}
As 
\[
T(U_{1,2} \oplus U_{1,3} \oplus U_{1,5})=  (x+y)(x+y+y^2)(x+y+y^2+y^3+y^4),
\]
one can find cases showing that the coefficients of $\bar\gamma$-elements could be negative.

\end{document}